\documentclass[a4paper, 11pt]{amsart}

\usepackage{amsmath, amsthm, amsfonts, amssymb}
\usepackage{enumerate}
\usepackage[bookmarks=false]{hyperref}
\usepackage{tikz-cd}
\usepackage{url}
\usepackage{xcolor}
\usepackage{verbatim}
\usepackage[normalem]{ulem}
\usepackage{mathtools}
\usepackage[
    backend=biber,
    style=alphabetic,
    sorting=nyt,
    maxbibnames=99
]{biblatex}

\title{Non-isomorphism of $q$-Gaussian von Neumann Algebras in Infinite Variables}

\author{Changying Ding}
\address{Department of Mathematics, University of California, Los Angeles, Los Angeles, CA 90095, USA}
\email{cding@math.ucla.edu}

\author{Zhiyuan Yang}
\address{Department of Mathematics, Purdue University, West Lafayette, IN 47907, USA}
\email{yang3261@purdue.edu}

\date{}

\newtheorem{theorem}{Theorem}[section]
\newtheorem{proposition}[theorem]{Proposition}
\newtheorem{lemma}[theorem]{Lemma}
\newtheorem{corollary}[theorem]{Corollary}

\newtheorem*{maintheorem}{Main Theorem}
\theoremstyle{definition}
\newtheorem{definition}[theorem]{Definition}
\newtheorem{remark}[theorem]{Remark}
\newtheorem{notation}[theorem]{Notation}
\theoremstyle{plain}

\numberwithin{equation}{section}

\newcommand{\bbR}{\mathbb{R}}
\newcommand{\bbX}{\mathbb{X}}

\newcommand{\bbC}{\mathbb{C}}

\newcommand{\C}{\mathcal{C}}
\newcommand{\F}{\mathcal F_q(H)}

\newcommand{\Ccal}{\mathcal C_q}
\newcommand{\Xcal}{\mathcal X_q}
\newcommand{\K}{\mathbb {K}}

\newcommand{\B}{\mathbb B}
\newcommand{\dist}{\operatorname{dist}}
\newcommand{\inv}{\operatorname{inv}}

\newcommand{\norm}[1]{\left\|#1\right\|}

\begin{document}

\begin{abstract}
We show that the \(q\)-Gaussian von Neumann algebras associated with an
infinite-dimensional real Hilbert space are pairwise nonisomorphic as
\(q\) ranges over \((-1,1)\).
\end{abstract}

\maketitle

\section{Introduction}

The $q$-Gaussian von Neumann algebras, introduced by Bo\.zejko and
Speicher \cite{BoSp91}\cite{BoSp94} (see also \cite{BKS97}), form a one-parameter
deformation of Voiculescu's free Gaussian construction \cite{VDN92}.
Given a real Hilbert space $H_{\bbR}$ and $q\in(-1,1)$, the
$q$-Gaussian construction produces a von Neumann algebra
$M_q(H_{\bbR})$. 
At $q=0$, this is the free group factor with
$\dim H_{\bbR}$ generators. 
For general $q$, Bo\.zejko, K{\"u}mmerer,
and Speicher \cite{BKS97} proved that $M_q(H_{\bbR})$ is a
non-injective $\mathrm{II}_1$ factor when $H_\bbR$ is infinite dimensional,
while
Ricard \cite{Ric05} and Nou \cite{Nou06} later showed that factoriality
and non-injectivity, respectively, hold whenever
$\dim H_{\bbR}\geq2$.

Understanding the dependence of the
isomorphism class of $M_q(H_{\bbR})$ on $q$ has been a longstanding
problem; see \cite[Section~1.4]{GuSh14},
\cite[Question~1.1]{NeZe18}, and \cite[Remark~4.7]{BCKW22}. When $H_{\bbR}$ is finite-dimensional,
Guionnet and Shlyakhtenko \cite{GuSh14}, using free monotone transport
together with conjugate-variable estimates of Dabrowski \cite{Dab14},
proved that $M_q(H_{\bbR})$ is isomorphic to the corresponding free
group factor whenever $|q|$ is sufficiently small, with the required
smallness depending on $\dim H_{\bbR}$.
More recently, Miyagawa and
Speicher \cite{MiSp23} established the existence of Lipschitz conjugate
systems for all $q\in(-1,1)$, but the finite-dimensional isomorphism
problem remains open beyond the small-$|q|$ regime.

The situation in infinite dimension is rather different. Building on
the $C^*$-algebraic result of Borst, Caspers, Klisse, and Wasilewski
\cite{BCKW22} and an idea of Ozawa \cite{Oza10}, Caspers \cite{Cas23}
proved that $M_q(H_{\bbR})$ is not isomorphic to $M_0(H_{\bbR})$
whenever $q\neq0$ and $H_{\bbR}$ is infinite-dimensional. His proof uses a
$\mathrm{W}^*$-version of Ozawa's condition (AO) \cite{Oza04}, which can be
understood through the canonical map
\[
\pi_q:C^*\bigl(M_q(H_{\bbR}),JM_q(H_{\bbR})J\bigr)
\longrightarrow
M_q(H_{\bbR})\otimes_{\min}JM_q(H_{\bbR})J.
\]
For $q=0$, the kernel of this map is contained in the space of
operators compact in the sense of Ozawa \cite{Oza10}, whereas for
$q\neq0$ it is strictly larger.

The preceding obstruction records only whether the kernel extends
beyond the compact, which separates $q=0$ from $q\neq0$ but cannot
distinguish between two nonzero values of $q$. We show that finer
structures of the kernels nevertheless determine the deformation
parameter completely.

\hypertarget{thm:main}{}
\begin{maintheorem}
Let $H_{\bbR}$ be an infinite-dimensional real Hilbert space and
$q,r\in(-1,1)$.
Then $M_q(H_{\bbR})$ is isomorphic to $M_r(H_{\bbR})$ if and only if
$q=r$.
\end{maintheorem}

The starting point of our argument is the following simple observation.
Let
$H=H_{\bbR}\oplus iH_{\bbR}$, let $\mathcal F_q(H)$ be the $q$-Fock
space, and denote by $P_n$ the projection onto $H_q^{\otimes n}$. The
projection $P_1$ already records the parameter $q$. Indeed, 
after removing the vacuum component,
multiplication on the range of $P_1$ gives the map
\[
Q_{P_1}:P_1\mathcal F_q(H)\otimes P_1\mathcal F_q(H)
=H\otimes H\ni h\otimes g
\longmapsto h\otimes g\in H_q^{\otimes2}\subset\mathcal F_q(H),
\]
which satisfies $Q_{P_1}^*Q_{P_1}=\mathrm{id}+qF$, where
$F:H\otimes H\ni h\otimes g\longmapsto g\otimes h$ is the flip map.
However, an isomorphism between $q$-Gaussian algebras need not preserve
$P_1$.

It turns out that the kernel of the preceding canonical map $\pi_q$, 
which we call the Akemann–Ostrand kernel, 
is preserved under isomorphisms and is generated by the projections
$P_n$ (Section~\ref{sec:akemann-ostrand-kernel}). 
Its $C^*$-analogue has a finer ideal structure:
the ideal generated by $P_1$ is the least 
noncompact ideal (Section~\ref{subsec:ideal-filtration}).
Consequently, although an isomorphism
$\theta:M_q(H_{\bbR})\to M_r(H_{\bbR})$ need not send $P_1^{(q)}$ to
$P_1^{(r)}$, it does send $P_1^{(q)}$ to a projection $e$ in the ideal
generated by $P_1^{(r)}$.

In fact, this containment retains enough information to recover $q$. The
definition of $Q_{P_1}$ extends to the class of ``first chaos projections'' in
the ideal generated by $P_1$ (Definition~\ref{def:first-chaos-projection}),
which, on the $r$-Gaussian side, includes $e$. Since $\theta$ transports the
identity for $Q_{P_1^{(q)}}$, the resulting map $Q_e$ also satisfies
$Q_e^*Q_e=\mathrm{id}+qF_e$. On the other hand, our estimate
(Proposition~\ref{prop:first-chaos-parameter} and
Appendix~\ref{sec:ordered-swap-estimate}) shows that every first chaos
projection $f$ in the ideal generated by $P_1^{(r)}$ must satisfy
$Q_f^*Q_f=\mathrm{id}+tF_f$ for some $t\in r[-r^2,1]$. Hence
$q\in r[-r^2,1]$, and by symmetry, $r\in q[-q^2,1]$, from which we conclude $q=r$.

\subsection*{Acknowledgments}
We thank Martijn Caspers, Akihiro Miyagawa, Dima Shlyakhtenko, and Thomas Sinclair for their helpful feedback.
C.D.\ is supported in part by NSF grant DMS-2554483.

\subsection*{AI use}

ChatGPT 5.5 and 5.6 Sol were used for editorial assistance and as interactive aids in exploring proof arguments.
In particular, AI-assisted
discussions contributed to the argument for Theorem~\ref{thm:length-filtration-main}(3) and to the idea of
using the operators \(V_\alpha\) to relate a general minimal projection to
\(P_1\) in the proof of the main theorem. ChatGPT 5.6 Sol was also used to
locate \cite{MS03} for the estimate in the appendix. Apart from these contributions, the mathematical ideas, framework, and constructions are due to the authors. 


\section{Preliminaries and Notations}

\subsection{$q$-Gaussian algebras}
\label{subsec:q-gaussian-algebras} Let $H_{\bbR}$ be a real Hilbert space, and let
$H=H_{\bbR}\oplus iH_{\bbR}$ be its complexification. We denote by
$J$ the conjugation on $H$, given by $J(\xi+i\eta)=\xi-i\eta$ for all
$\xi,\eta\in H_{\bbR}$. For $q\in(-1,1)$, the $q$-Fock space is defined as
$\F=\bbC\Omega\oplus\bigoplus_{n=1}^{\infty}H_q^{\otimes n}$, where $H_q^{\otimes n}$ denotes the completion of
$H^{\otimes_{\mathrm{alg}}n}$ with respect to the $q$-inner product
\[
 \left\langle \xi_1\otimes\cdots\otimes\xi_n,
 \eta_1\otimes\cdots\otimes\eta_n\right\rangle_q
 :=\sum_{\sigma\in S_n}q^{\inv(\sigma)}
 \prod_{i=1}^n\left\langle\xi_i,\eta_{\sigma(i)}\right\rangle,
\]
where $\inv(\sigma)$ is the inversion number of $\sigma$.
For $n\geq0$, we will denote by $P_n\in\B(\F)$ the orthogonal projection onto
$H_q^{\otimes n}$, where $H_q^{\otimes0}=\bbC\Omega$. The conjugation
$J$ extends to $\F$ by
$J(\xi_1\otimes\cdots\otimes\xi_n)
=J\xi_n\otimes\cdots\otimes J\xi_1$.

We say that $\xi\in\F$ has length $n$, and write $|\xi|=n$, if
$\xi\in H_q^{\otimes n}$. And, we call $\xi$ an elementary tensor if
$\xi\in\bbC\Omega$ or $\xi=h_1\otimes\cdots\otimes h_n$ for some
$h_1,\ldots,h_n\in H$.

For $h\in H$, the left creation operator $\ell(h)$ is defined by
\[
 \ell(h)\Omega=h,
 \qquad
 \ell(h)(h_1\otimes\cdots\otimes h_n)
 =h\otimes h_1\otimes\cdots\otimes h_n.
\]
It extends to a bounded operator on $\F$, and its adjoint
$\ell(h)^\ast$ is called the left annihilation operator. The conjugation $J$
turns left creation into right creation: put $r(h)=J\ell(Jh)J$, so that
\[
 r(h)\Omega=h,
 \qquad
 r(h)(h_1\otimes\cdots\otimes h_n)
 =h_1\otimes\cdots\otimes h_n\otimes h.
\]
Also, we denote the field operator by $s(h):=\ell(h)+\ell(Jh)^\ast$ for $h\in H$. For $h\in H_{\bbR}$, one has $Jh=h$, and hence
$s(h)=\ell(h)+\ell(h)^\ast$ is self-adjoint. Recall that we have the norm estimates
\[
 \norm{\ell(h)}=\max\{1,(1-q)^{-1/2}\}\norm{h}
 \quad(h\in H),
 \qquad
 \norm{s(h)}=\frac{2}{\sqrt{1-q}}\norm{h}
 \quad(h\in H_{\bbR}),
\]
see \cite[Lemma~4]{BoSp91} and
\cite[Theorem~1.10(1)]{BKS97}.

The $q$-Gaussian $C^\ast$-algebra is defined as
$A_q(H_{\bbR})=C^\ast(s(h):h\in H_{\bbR})$
, and the $q$-Gaussian von Neumann algebra is defined as
$M_q(H_{\bbR})=A_q(H_{\bbR})''$. The vacuum vector $\Omega$ defines
a faithful normal tracial state
\[
 \tau_q(x)=\langle x\Omega,\Omega\rangle
 \qquad (x\in M_q(H_{\bbR})).
\]
In particular, $\Omega$ is cyclic and separating, and the map
$x\mapsto x\Omega$ identifies
$L^2(M_q(H_{\bbR}),\tau_q)$ with $\mathcal F_q(H)$; under this
identification, the modular conjugation is precisely the
conjugation $J$ defined above. See
\cite[Proposition~2.3]{BKS97}.

For an elementary tensor
\(\xi=h_1\otimes\cdots\otimes h_k\), we will also denote
\(\ell(\xi)=\ell(h_1)\cdots\ell(h_k)\) and
\(r(\xi)=r(h_k)\cdots r(h_1)=J\ell(J\xi)J\), and extend these
definitions linearly to \(H^{\otimes_{\mathrm{alg}}k}\). For every
\(\xi\in H^{\otimes_{\mathrm{alg}}n}\), the Wick operator \(W(\xi)\) is the unique operator in $A_q(H_{\bbR})$ such that \(W(\xi)\Omega=\xi\). The resulting
Wick map extends continuously from \(H^{\otimes_{\mathrm{alg}}n}\) to
\(H_q^{\otimes n}\); for \(\xi\in H_q^{\otimes n}\), we continue to
write \(W(\xi)\in A_q(H_{\bbR})\) for its value. We also put
\(W_R(\xi)=JW(J\xi)J\), so that \(W_R(\xi)\Omega=\xi\). See
\cite[Definition~2.5 and Proposition~2.7]{BKS97}.

Put \(Q=\sum_{n\geq0}q^nP_n\). For \(f,g\in H\), the following commutation
relations are standard, and will be used repeatedly throughout the paper:
\begin{align*}
 \ell(f)^*\ell(g)
 &=\langle f,g\rangle 1+q\ell(g)\ell(f)^*,
 &
 r(f)^*r(g)
 &=\langle f,g\rangle 1+qr(g)r(f)^*,
 \\
 \ell(f)^*r(g)
 &=r(g)\ell(f)^*+\langle f,g\rangle Q,
 &
 r(f)^*\ell(g)
 &=\ell(g)r(f)^*+\langle f,g\rangle Q,
 \\
 Q\ell(f)&=q\ell(f)Q,
 &
 Qr(f)&=qr(f)Q.
\end{align*}
For \(n\geq0\), with \(P_j=0\) for \(j<0\),
\begin{equation*}
 P_n\ell(f)=\ell(f)P_{n-1},
 \qquad
 P_nr(f)=r(f)P_{n-1}.
\end{equation*}

We next record some elementary identities about norm limit. See the proof of \cite[Theorem 2.1]{BKS97} for the weak limit version of similar identities.

\begin{lemma}\label{lem:length-averages}
Let \(H_{\bbR}\) be an infinite-dimensional real Hilbert space
and \(q\in(-1,1)\setminus\{0\}\). Choose an orthonormal family
\((h_i)_{i\geq1}\) in \(H_{\bbR}\), and put
\(\ell_i=\ell(h_i)\), \(r_i=r(h_i)\), and
\(s_i=\ell_i+\ell_i^\ast\). Set
\(C=C^\ast(A_q(H_{\bbR}),JA_q(H_{\bbR})J)\). Then the following hold:
\begin{enumerate}
 \item\label{item:length-averages-projections}
 $\lim_{N\to \infty}\tfrac{1}{N}\sum_{i=1}^Ns_iJs_iJ = Q$ in norm. In particular, \(P_n\in C\) for every \(n\geq0\).
 \item\label{item:length-averages-convergence}
 For every \(n\geq1\),
 \(\lim_{N\to\infty}\frac1N\sum_{i=1}^Ns_iP_nJs_iJ
 =q^{n-1}P_{n-1}\) in norm.
\end{enumerate}
\end{lemma}
\begin{proof}
Put \(c_q=(1-|q|)^{-1}\). By \cite[Lemma~2.1]{Hia03} (see also
\cite[Lemma~A.2]{Yan24}),
\(
 \norm{\sum_{i=1}^N\ell_i\ell_i^*}\leq c_q.
\)
The same estimate holds with \(r_i\) in place of \(\ell_i\). On the
other hand, the creation-operator bound above gives
\(
 \norm{\sum_{i=1}^N\ell_i^*\ell_i},
 \ \norm{\sum_{i=1}^Nr_i^*r_i}
 \leq Nc_q.
\)
Thus the row operators
\[
 R_\ell=[\ell_1\ \cdots\ \ell_N],\qquad
 R_r=[r_1\ \cdots\ r_N]:\F^{\oplus N}\longrightarrow\F
\]
satisfy \(\norm{R_\ell},\norm{R_r}\leq c_q^{1/2}\); for \(R_r\), use
\(r_i=J\ell_iJ\). If
\(C_\ell=[\ell_1\ \cdots\ \ell_N]^{\mathsf t}\) and
\(C_r=[r_1\ \cdots\ r_N]^{\mathsf t}\), then the creation-operator
bound above gives
\(\norm{C_\ell},\norm{C_r}\leq(Nc_q)^{1/2}\).

For~\emph{(1)}, let \(N\geq1\) and put
\(S_N=N^{-1}\sum_{i=1}^Ns_iJs_iJ\in C\).  Since
\(\ell_i^*r_i=r_i\ell_i^*+Q\), one has
\[
 S_N-Q=\frac1N\sum_{i=1}^N
 \bigl(\ell_ir_i+\ell_ir_i^*+r_i\ell_i^*+\ell_i^*r_i^*\bigr).
\]
We have
\begin{align*}
 \norm{\sum_{i=1}^N\ell_ir_i}
 &\leq\norm{R_\ell}\norm{C_r}\leq c_q\sqrt N,
 &
 \norm{\sum_{i=1}^N\ell_ir_i^*}
 &\leq\norm{R_\ell}\norm{R_r}\leq c_q.
\end{align*}
The sum involving \(\ell_i^*r_i^*\) obeys the first bound, while the
sum involving \(r_i\ell_i^*\) obeys the second. Hence
\[
 \norm{S_N-Q}
 \leq2c_q\bigl(N^{-1/2}+N^{-1}\bigr)\longrightarrow0.
\]
Consequently, \(Q\in C\). Since \(q^n\) is isolated in
\(\sigma(Q)=\{0\}\cup\{q^k:k\geq0\}\), continuous functional calculus
gives \(P_n\in C^*(Q)\subset C\).

For~\emph{(2)}, fix \(n\geq1\) and \(N\geq1\). Note that
\(\ell_i^*P_nr_i=\ell_i^*r_iP_{n-1}
=(r_i\ell_i^*+Q)P_{n-1}
=r_i\ell_i^*P_{n-1}+q^{n-1}P_{n-1}\).
Expanding \(s_iP_nJs_iJ\), it follows that
\[
 s_iP_nJs_iJ-q^{n-1}P_{n-1}
 =\ell_iP_nr_i+\ell_iP_nr_i^*
  +r_i\ell_i^*P_{n-1}+\ell_i^*P_nr_i^*.
\]

Since \(P_n\) and \(P_{n-1}\) are contractions, we have
\begin{align*}
 \norm{\sum_{i=1}^N\ell_iP_nr_i}
 &\leq\norm{R_\ell}\norm{C_r}\leq c_q\sqrt N,
 &
 \norm{\sum_{i=1}^N\ell_iP_nr_i^*}
 &\leq\norm{R_\ell}\norm{R_r}\leq c_q.
\end{align*}
The sum involving \(\ell_i^*P_nr_i^*\) obeys the first bound, while
the sum involving \(r_i\ell_i^*P_{n-1}\) obeys the second.
Consequently,
\[
 \norm{
  \frac1N\sum_{i=1}^Ns_iP_nJs_iJ
  -q^{n-1}P_{n-1}}
 \leq2c_q\bigl(N^{-1/2}+N^{-1}\bigr)\longrightarrow0.
\]
\end{proof}

\subsection{Boundary pieces for von Neumann algebras}

We recall the notion of boundary pieces from \cite{DKEP22,DP23}.
Let $M$ be a von Neumann algebra and let $J$ denote its canonical
conjugation on $L^2M$. A nonzero hereditary $C^\ast$-subalgebra
$\bbX\subset\B(L^2M)$ is called an $M$-boundary piece if
$\mathcal M(\bbX)\cap M$ and $\mathcal M(\bbX)\cap JMJ$ are
ultraweakly dense in $M$ and $JMJ$, respectively, where
$\mathcal M(\bbX)$ denotes the multiplier algebra of $\bbX$. For
$T\in\B(L^2M)$, consider the $L^\infty-L^1$ norm $\|T\|_{\infty,1}=\sup_{a,b\in(M)_1}
|\langle T\widehat a,\widehat b\rangle|$.
For an $M$-boundary piece $\bbX$, following \cite{DKEP22}, we denote 
$\mathbb K_{\bbX}^{\infty,1}(M) = \overline{\bbX}^{\infty,1}$ the $\|\cdot\|_{\infty,1}$-closure
of $\bbX$.
When discussing relative biexactness and the small-at-infinity boundary, we use the notation \(\mathbb K_{\bbX}^{\infty,1}(M)\) to conform with the literature. Elsewhere, when only the \(\|\cdot\|_{\infty,1}\)-closure of \(\bbX\) is relevant, we use the simpler notation \(\overline{\bbX}^{\infty,1}\). Note that although $\mathbb K_{\bbX}^{\infty,1}(M)$ is not an algebra
in general, it is invariant under left and right multiplication by $M$
and $JMJ$.

Following \cite{DKEP22,DP23}, define the small-at-infinity
boundary relative to $\bbX$ by the operator system
\[
 \mathbb S_{\bbX}(M)
 =\{T\in\B(L^2M):[T,x]\in\mathbb K_{\bbX}^{\infty,1}(M)
       \text{ for every }x\in JMJ\}.
\]
We say that $M$ is \emph{biexact relative to $\bbX$} if the inclusion
$M\subset\mathbb S_{\bbX}(M)$ is $M$-nuclear. Here $M$-nuclearity
means nuclearity with respect to the weak $M$-topology induced by the
states $\varphi\in\B(L^2M)^\ast$ that are normal when restricted to $M$. See \cite{DP23} for details about the definition of biexact von Neumann algebra.

We now specialize to the $q$-Gaussian setting. Assume that
$H_{\bbR}$ is separable and has infinite dimension, and that $q\ne0$.
Put in short $A_q=A_q(H_{\bbR})$ and $M=M_q(H_{\bbR})$. 
The canonical boundary piece of $M$ we are interested in is
\[
 \bbX_q
 =\overline{\operatorname{span}}^{\|\cdot\|}
 \left\{
 P_kTP_\ell:
 k,\ell\geq0,\;
 T\in\B(\mathcal F_q(H))
 \right\}.
\]

Equivalently, $\bbX_q$ is the hereditary $C^\ast$-subalgebra of
$\B(\mathcal F_q(H))$ generated by the length projections $P_n$. Note that
the projections $P_{\leq n}=\sum_{k=0}^nP_k$ form an approximate identity
for $\bbX_q$. Moreover, $A_q,JA_qJ\subset\mathcal M(\bbX_q)$, where
$A_q\subset M$ is the $q$-Gaussian $C^\ast$-algebra, and thus
$\bbX_q$ is an $M_q$-boundary piece. The other boundary piece we need for $M_q$ is $\K(L^2M_q)=\K(\F)$.

For a finite von Neumann algebra \(M\), we say a projection
\(P\in\B(L^2M)\) \emph{\(M\)-regular} if it satisfies one of the
equivalent conditions in the following basic lemma. Denote $ \widehat{M}= M\hat{1}\subset L^2M $.

\begin{lemma}
\label{lem:regular-corners}
Let \(M\) be a finite von Neumann algebra and
\(P\in\B(L^2M)\) a projection.  The following are equivalent:
\begin{enumerate}
\item \(P(L^2M)\subset\widehat M\);
\item \(P(L^2M)\subset\widehat M\), and
\(P:L^2M\to(\widehat M,\norm{\cdot}_\infty)\) is bounded;
\item \(P_{\mid\widehat M}:(\widehat M,\norm{\cdot}_1)
\to(L^2M,\norm{\cdot}_2)\) is bounded.
\end{enumerate}
Moreover, if \(P\) is \(M\)-regular, then the following maps are
continuous:
\[
 \begin{aligned}
 (M,\norm{\cdot}_2)\ni x
 &\mapsto xP\in\bigl(\B(L^2M),\norm{\cdot}\bigr),\\
 \bigl(\B(L^2M),\norm{\cdot}_{\infty,1}\bigr)\ni T
 &\mapsto PTP\in\bigl(\B(L^2M),\norm{\cdot}\bigr).
 \end{aligned}
\]
\end{lemma}

\begin{proof}
Clearly, (2) implies (1), while (1) implies (2) by the closed graph
theorem.
Suppose that the map in (2) has norm at most \(c\).  For \(x\in M\),
one has
\[
 \norm{P\widehat{x}}_2
 =\sup_{\norm{\xi}_2\leq1}|\langle\widehat{x},P\xi\rangle|
 \leq c\sup_{\norm{y}_\infty\leq1}
 |\langle\widehat{x},\widehat{y}\rangle|
 =c\norm{x}_1.
\]
Hence (3) holds.  Conversely, if the map in (3) has norm at most
\(c\), then, for \(\xi\in L^2M\) and \(x\in M\),
\(
 |\langle P\xi,\widehat{x}\rangle|
 =|\langle\xi,P\widehat{x}\rangle|
 \leq c\norm{\xi}_2\norm{x}_1.
\)
It follows that \(P\xi\in\widehat M\) and
\(\norm{P\xi}_\infty\leq c\norm{\xi}_2\), so (2) holds.

Assume now that \(P\) is \(M\)-regular, and let \(c\) be the common
norm of the maps in (2) and (3).  For \(x\in M\),
\[
 \norm{xP}
 =\sup_{\norm{\xi}_2\leq1}\norm{xP\xi}_2
 \leq\norm{x}_2\sup_{\norm{\xi}_2\leq1}\norm{P\xi}_\infty
 \leq c\norm{x}_2.
\]
Finally, for \(T\in\B(L^2M)\),
\(
 \norm{PTP}
 =\sup_{\norm{\xi}_2,\norm{\eta}_2\leq1}
 |\langle TP\xi,P\eta\rangle|
 \leq c^2\norm{T}_{\infty,1}.
\)
\end{proof}

We now collect the notation used throughout the paper.

\begin{notation}\label{not:q-gaussian-notation}
Let $H_{\bbR}$ be a separable infinite-dimensional real Hilbert space, $H=H_{\bbR}\oplus iH_{\bbR}$, and fix
$q\in(-1,1)\setminus\{0\}$. Let $\mathcal F_q(H)$ be the $q$-Fock
space, put $M_q=M_q(H_{\bbR})$ and $A_q=A_q(H_{\bbR})$, and identify
$L^2(M_q)$ with $\mathcal F_q(H)$, with canonical conjugation $J$.
For $n\geq0$, let $P_n^{(q)}$ be the projection onto
$H_q^{\otimes n}$. Finally, recall that
\[
 \bbX_q
 =\overline{\operatorname{span}}^{\|\cdot\|}
 \{P_k^{(q)}TP_\ell^{(q)}:k,\ell\geq0,\
 T\in\B(\mathcal F_q(H))\}
\]
is the $M_q$-boundary piece introduced above.
Inside $\B(\mathcal F_q(H))$, we use the following notation.

\noindent
\begin{minipage}[t]{0.48\textwidth}
\vspace{0pt}
\begin{itemize}
 \renewcommand{\labelitemi}{\textendash}
 \item $C_q=C^\ast(A_q,JA_qJ)$.
 \item $J_n^{(q)}=
 \overline{C_qP_n^{(q)}C_q}^{_{\|\cdot\|}}$ for $n\geq0$.
 \item $X_q=\overline{\bigcup_{n\geq0}J_n^{(q)}}^{_{\|\cdot\|}}$.
\end{itemize}
\end{minipage}\hfill
\begin{minipage}[t]{0.48\textwidth}
\vspace{0pt}
\begin{itemize}
 \renewcommand{\labelitemi}{\textendash}
 \item $\mathcal C_q=C^\ast(M_q,JM_qJ)$.
 \item $\mathcal X_q=\mathcal C_q\cap
 \overline{\bbX_q}^{_{\|\cdot\|_{\infty,1}}}$.
 \item $\mathcal K_q=\mathcal C_q\cap
 \overline{\K(\mathcal F_q(H))}^{_{\|\cdot\|_{\infty,1}}}
 \subset\mathcal X_q$.
\end{itemize}
\end{minipage}

We record that 
$P_n^{(q)}\in C_q$ and hence
$J_n^{(q)}\lhd C_q$, for every $n\geq0$;
$X_q=C_q\cap\bbX_q$;
and $ \mathcal X_q $ and $ \mathcal K_q $ are both ideals of $ \mathcal C_q $.

The first assertion follows from
part~(\ref{item:length-averages-projections}) of
Lemma~\ref{lem:length-averages}. For the second,
part~(\ref{item:length-averages-convergence}) of
Lemma~\ref{lem:length-averages} gives
$P_{\leq n}^{(q)}:=\sum_{k=0}^nP_k^{(q)}\in J_n^{(q)}$.
Since $P_{\leq n}^{(q)}$ is an approximate identity for $\bbX_q$,
every $x\in C_q\cap\bbX_q$ is the norm limit of
$P_{\leq n}^{(q)}xP_{\leq n}^{(q)}\in J_n^{(q)}$.
Conversely, $C_q\subset\mathcal M(\bbX_q)$, so
$J_n^{(q)}\subset C_q\cap\bbX_q$. 
Lastly, notice that  $ \overline{\bbX_q}^{_{\|\cdot\|_{\infty,1}}} $ 
 and $ \overline{\K(\mathcal F_q(H))}^{_{\|\cdot\|_{\infty,1}}}$ 
 contains both $M_q$ and $JM_qJ$ in their multipliers.

We suppress $q$ from the notation
whenever no confusion can arise.
\end{notation}

\section{Computing the Akemann--Ostrand kernels}
\label{sec:akemann-ostrand-kernel}

Our strategy towards the main theorem begins with the observation that an
isomorphism $M_q\cong M_r$ identifies the corresponding
Akemann--Ostrand kernels. The first step is therefore to compute this
kernel for $M_q$. 
Using the machinery of \cite{DP23} and the proof of \cite{Oza10},
we prove in this section that it is precisely
$\mathcal X_q=\mathcal C_q\cap
\overline{\bbX_q}^{_{\|\cdot\|_{\infty,1}}}$.

\begin{theorem}\label{thm:akemann-ostrand-kernel}
Let
$q\in(-1,1)\setminus\{0\}$. The map
\[
 \mathcal C_q=C^\ast(M_q,JM_qJ)\ni xJyJ
 \longmapsto x\otimes JyJ
 \in M_q\otimes_{\min}JM_qJ
\]
extends to a surjective \(\ast\)-homomorphism \(\pi_q\) with
\(\ker\pi_q=\mathcal X_q\).
\end{theorem}

We refer to \cite{DP23} for the basic properties and equivalent
formulations of relative biexactness.
Theorem~\ref{thm:akemann-ostrand-kernel} follows
from the following relative biexactness result.

\begin{proposition}\label{prop:relative-biexactness}
Let $q\in(-1,1)\setminus\{0\}$ and $H_{\bbR}$ be a separable
infinite-dimensional real Hilbert space. Then the $q$-Gaussian von
Neumann algebra $M_q(H_{\bbR})$ is biexact relative to the boundary
piece $\bbX_q$.
\end{proposition}

\begin{proof}
Put $M=M_q(H_{\bbR})$ and choose increasing finite-dimensional real
subspaces $F_{d,\bbR}\subset H_{\bbR}$ with dense union, setting
$M_d=M_q(F_{d,\bbR})$. Let $E_d:M\to M_d$ be the canonical
trace-preserving conditional expectation, and note that
$E_d\to\operatorname{id}_M$ pointwise in $\|\cdot\|_2$. We shall use \cite[Theorem~7.3]{DP23} to prove the relative biexactness.

Let $H_{\bbR}^{(0)}$ and $H_{\bbR}^{(1)}$ be two orthogonal copies of
$H_{\bbR}$, and let $H^{(0)}$ and $H^{(1)}$ denote their complexifications.
Put
$\widetilde M=M_q(H_{\bbR}^{(0)}\oplus H_{\bbR}^{(1)})$.  Identify
$M=M_q(H_{\bbR}^{(0)}\oplus0)\subset\widetilde M$.  For $k\geq0$, let
$\mathcal L_k\subset L^2(\widetilde M)$ be the closed span of the
tensors containing exactly $k$ letters from $H^{(1)}$, and denote by
$R_k$ the orthogonal projection onto $\mathcal L_k$.  By
\cite[Proposition~8.10 and the proof of Theorem~8.11]{DP23}, using
\cite[Proposition~4.1]{Av11}, for each $d$ there is an integer
$m_d\geq1$ such that the $M_d$--$M_d$ bimodule
$\mathcal H_d=Q_dL^2(\widetilde M)$ is weakly coarse, where
$Q_d=\sum_{k\geq m_d}R_k$.

For $0<t<\pi/2$, let
\[
 U_t=
 \begin{pmatrix}
  \cos(t)&-\sin(t)\\
  \sin(t)&\cos(t)
 \end{pmatrix}
 \in\mathcal O\bigl(H_{\bbR}^{(0)}\oplus H_{\bbR}^{(1)}\bigr),
\]
and let $\alpha_t\in\operatorname{Aut}(\widetilde M)$ be its second
quantization, so that $\|\alpha_t(x)-x\|_2\to0$ as $t\to0$ for every
$x\in M$.  The isometry $V_t:L^2(M)\to L^2(\widetilde M)$,
$V_t\widehat x=\widehat{\alpha_t(x)}$, defines a u.c.p.\ map
\[
 \B(L^2\widetilde M)\ni T\longmapsto
 \Phi_t(T)=V_t^*TV_t\in\B(L^2M).
\]

For $k\geq0$, a direct computation shows
that
\[
 \Phi_t(R_k)P_n
 =\binom nk\cos^{2(n-k)}(t)\sin^{2k}(t)P_n,
\]
where $\binom nk=0$ for $n<k$.  Consequently,
$\Phi_t(R_k)=\sum_{n\geq k}\binom nk
\cos^{2(n-k)}(t)\sin^{2k}(t)P_n$.  Since
\[
 \sum_{n\geq k}\binom nk\cos^{2(n-k)}(t)\sin^{2k}(t)
 =\sin^{2k}(t)\sum_{j\geq0}\binom{j+k}{k}\cos^{2j}(t)
 =\frac1{\sin^2(t)},
\]
the series converges in operator norm and hence
$\Phi_t(R_k)\in\bbX_q$.

\begin{samepage}
We now construct the u.c.p.\ approximation.  Let $e_d$ be the
projection from $L^2M$ onto $L^2M_d$.  Compression by $e_d$ and the
weak containment above give an $M_d$-modular u.c.p.\ map (cf.~\cite[proof of Theorem~7.8]{DP23})
\[
 \Psi_d:\B(L^2M)\longrightarrow
 (JM_dJ)'\cap\B(\mathcal H_d)
\]
such that $\Psi_d(1)=Q_d$ and
$\Psi_d(x)=E_d(x)Q_d=Q_dE_d(x)$ for $x\in M$, where
$\B(\mathcal H_d)=Q_d\B(L^2\widetilde M)Q_d$.  Moreover,
$K_{d,t}:=\Phi_t(1-Q_d)
=\sum_{k=0}^{m_d-1}\Phi_t(R_k)\in\bbX_q$.
\end{samepage}

Fix a state $\omega$ on $\B(L^2M)$ and define a u.c.p.\ map
\[
 \Theta_{d,t}(T)=\Phi_t(\Psi_d(T))+\omega(T)K_{d,t},
 \qquad T\in\B(L^2M).
\]

For $x,y\in M$, write $x_d=E_d(x)$ and $y_d=E_d(y)$.  We claim that
\begin{equation}\label{eq:approximation-estimate}
 \dist_{\|\cdot\|_{\infty,1}}
 \bigl(\Theta_{d,t}(x)-x,
       \mathbb K_{\bbX_q}^{\infty,1}(M)\bigr)
 \leq\|x_d-\alpha_t(x_d)\|_2+\|x_d-x\|_2.
\end{equation}
Moreover, for every $T\in\B(L^2M)$,
\begin{equation}\label{eq:commutator-estimate}
 \dist_{\|\cdot\|_{\infty,1}}
 \bigl([\Theta_{d,t}(T),JyJ],
       \mathbb K_{\bbX_q}^{\infty,1}(M)\bigr)\leq
 2\|T\|\bigl(\|y_d-\alpha_t(y_d)\|_2+\|y-y_d\|_2\bigr).
\end{equation}

Assuming these estimates, fix $\varepsilon>0$ and let
$E\subset F\subset\B(L^2M)$ be finite-dimensional operator systems
with $E\subset M$. Choose $d$ and then $t>0$ so that
\[
 \sup_{x\in(E)_1}\norm{E_d(x)-x}_2<\frac{\varepsilon}{16},
 \qquad
 \sup_{x\in(E)_1}\norm{\alpha_t(E_d(x))-E_d(x)}_2
 <\frac{\varepsilon}{16}.
\]
Let \(r_\tau\) denote the factorization seminorm from
\cite[Section~3.2,Theorem~7.3]{DP23}, and write
\(d_{r_\tau}(T,\mathcal I)=\inf_{A\in\mathcal I}r_\tau(T-A)\).
By \cite[Proposition~3.1]{DKEP22}, one has
$d_{r_\tau}\leq4d_{\norm{\cdot}_{\infty,1}}$. Hence equations
\eqref{eq:approximation-estimate} and
\eqref{eq:commutator-estimate} give
\[
 d_{r_\tau}\bigl(\Theta_{d,t}(x)-x,
 \mathbb K_{\bbX_q}^{\infty,1}(M)\bigr)<\varepsilon,
 \qquad
 d_{r_\tau}\bigl([JxJ,\Theta_{d,t}(T)],
 \mathbb K_{\bbX_q}^{\infty,1}(M)\bigr)<\varepsilon
\]
for every $x\in(E)_1$ and $T\in(F)_1$. Moreover, $M$ is weakly exact
by \cite[Theorem~5.1]{KN10} and \cite[Proposition~4.1.2]{Iso12}.
Therefore \cite[Theorem~7.3]{DP23} shows that $M$ is biexact relative
to $\bbX_q$.

It remains to prove the two estimates.  For $x\in M$ and
$T\in\B(L^2\widetilde M)$, the same computation as in
\cite[Proof of Proposition~8.2]{DKEP22} gives
\begin{align*}
 \|\Phi_t(Tx)-\Phi_t(T)x\|_{\infty,1},\quad
 \|\Phi_t(xT)-x\Phi_t(T)\|_{\infty,1}
 &\leq\|T\|\,\|x-\alpha_t(x)\|_2,\\
 \|\Phi_t(TJxJ)-\Phi_t(T)JxJ\|_{\infty,1},\quad
 \|\Phi_t(JxJT)-JxJ\Phi_t(T)\|_{\infty,1}
 &\leq\|T\|\,\|x-\alpha_t(x)\|_2.
\end{align*}

Fix $x\in M$ and write $x_d=E_d(x)$.  Since
$\Psi_d(x)=x_dQ_d$, the first pair of inequalities gives
\[
 \|\Phi_t(\Psi_d(x))-x_d\Phi_t(Q_d)\|_{\infty,1}
 \leq\|x_d-\alpha_t(x_d)\|_2.
\]
Since $\Phi_t(Q_d)=1-K_{d,t}$, we have
\[
 \Theta_{d,t}(x)-x
 =\bigl(\Phi_t(\Psi_d(x))-x_d\Phi_t(Q_d)\bigr)
  +(x_d-x)+(\omega(x)1-x_d)K_{d,t}.
\]
As the last term belongs to $\mathbb K_{\bbX_q}^{\infty,1}(M)$, this
proves \eqref{eq:approximation-estimate}.

Similarly, fix $y\in M$ and $T\in\B(L^2M)$, and write
$y_d=E_d(y)$.  Since $\Psi_d(T)$ commutes with $Jy_dJ$, the second
pair of inequalities gives
\[
 \|[\Phi_t(\Psi_d(T)),Jy_dJ]\|_{\infty,1}
 \leq2\|T\|\,\|y_d-\alpha_t(y_d)\|_2.
\]
Moreover,
$\|[\Phi_t(\Psi_d(T)),J(y-y_d)J]\|_{\infty,1}
\leq2\|T\|\,\|y-y_d\|_2$, while
$\omega(T)[K_{d,t},JyJ]$ belongs to
$\mathbb K_{\bbX_q}^{\infty,1}(M)$; together these give the claimed
bound, which is \eqref{eq:commutator-estimate}.
\end{proof}

\begin{samepage}
\begin{proof}[Proof of Theorem~\ref{thm:akemann-ostrand-kernel}]
Recall that
$\mathbb K_{\bbX_q}^{\infty,1}(M)
=\overline{\bbX_q}^{_{\|\cdot\|_{\infty,1}}}$, and hence
$\Xcal=\Ccal\cap\mathbb K_{\bbX_q}^{\infty,1}(M)$.
By Proposition~\ref{prop:relative-biexactness} and
\cite[Theorem~7.19]{DP23}, the multiplication map extends to a unital
\(\ast\)-homomorphism
\[
 \nu_q:M\otimes_{\min}JMJ\ni x\otimes JyJ\longmapsto
 xJyJ+\Xcal\in\Ccal/\Xcal.
\]
Note that \(\nu_q\) is surjective and that \(M\) and \(JMJ\) are simple
\(C^\ast\)-algebras.  Hence \(M\otimes_{\min}JMJ\) is simple by
\cite[Corollary]{Tak64}.  Thus \(\nu_q\) is a
\(\ast\)-isomorphism, and we identify
\(M\otimes_{\min}JMJ=\Ccal/\Xcal\) through \(\nu_q\), from which the
conclusion follows.
\end{proof}
\end{samepage}

\begin{remark}
    We note that if we consider the $C^*$-analogy of the Akemann-Ostrand map $ C_q\to A_q\otimes_{\min }JA_qJ $, then one can show that its kernel is precisely $X_q = \overline{ \bigcup_{n\geq 0} J_n^{(q)} }$ assuming the nuclearity of the $q$-CCR algebra $ C^*( \ell(h):h\in H )  $, following the proof of \cite[Theorem~4.2]{Shl04}. (The nuclearity of $ C^*( \ell(h):h\in H )  $ can be shown for $|q|<\sqrt{2}-1$ by \cite{JSW94}, and more generally for $|q|<0.44$ by \cite{DN93} with an inductive limit argument.) Therefore, $ X_q $ can be considered as the $C^*$-analogy of $ \mathcal X_q $. Since this fact will not be used below, we omit its proof.
\end{remark}

\section{Structure of a $C^\ast$-analogue of the Akemann--Ostrand kernel}
\label{sec:ideal-filtration-first-layer}

\subsection{An ideal filtration}
\label{subsec:ideal-filtration}

We begin by studying the ideal structure of \(X=X_q\). In this section, we show that the length projections
\(P_n\) determine a natural filtration by the ideals \(J_n\), whose first
layer will later play a central role in recovering the parameter \(q\).

\begin{theorem}\label{thm:length-filtration-main} The ideals $J_n$ satisfies these properties:
\begin{enumerate}
\item\label{item:length-filtration-chain}
One has
\(\K(\F)=J_0\subset J_1\subset J_2\subset\cdots\).

\item\label{item:length-filtration-minimal-projections}
For every \(n\geq1\), \(p_n=P_n+J_{n-1}\) is a nonzero
minimal projection in \(C/J_{n-1}\).  Moreover, there exists
a Hilbert space \(E_n\) such that
\(J_n/J_{n-1}\cong\K(E_n)\).

\item\label{item:length-filtration-least-ideal}
For every \(n\geq1\), \(J_n/J_{n-1}\) is the least
nonzero ideal of \(X/J_{n-1}\).
\end{enumerate}
\end{theorem}

\begin{lemma}\label{lem:length-filtration}
We have the following:
\begin{enumerate}
\item\label{item:length-filtration-increasing}

\(\K(\F)=J_0\subset J_1\subset J_2\subset\cdots\).

\item\label{item:length-filtration-dense-subalgebra}
For every \(n\geq0\), the set \(D_n\) is a dense
\(\ast\)-subalgebra of \(J_n\), where
\[
 D_n=\operatorname{span}\left\{
 \ell(\xi)r(\eta)P_k\ell(\xi')^*r(\eta')^*:0\leq k\leq n,
 \ \xi,\eta,\xi',\eta'\in
 \bigcup_{m\geq0}H^{\otimes_{\mathrm{alg}}m}\right\}.
\]

\item\label{item:length-filtration-approximate-identity}
The sequence \(P_{\leq N}=\sum_{k=0}^NP_k\), \(N\geq0\), is an
approximate identity for \(X\).
\end{enumerate}
\end{lemma}

\begin{proof}
\emph{(1)} This is immediate from
part~(\ref{item:length-averages-convergence}) of
Lemma~\ref{lem:length-averages} and the definition of \(J_0\).

\emph{(2)} Note that for elementary tensors \(\xi,\eta\), if
\(j=|\xi|+|\eta|+k\), then
\(P_jW(\xi)W_R(\eta)P_k=\ell(\xi)r(\eta)P_k\).  Thus
part~(1) gives \(D_n\subset J_n\).

Let \(x,y\) be polynomials in \(s(h)\) and \(Js(h)J\),
where \(h\in H_{\bbR}\).  Observe that
the commutation relations between $P_n$, $\ell(h) $, and $r(g) $ show that
\(xP_ny\in D_n\).
It follows that \(D_n\) is dense in \(J_n\).  The same
relations show that \(D_n\) is a \(\ast\)-subalgebra.

\emph{(3)} For every \(T\in D_n\), one has
\(P_{\leq N}T=T=TP_{\leq N}\) for all sufficiently large \(N\).
Since \(P_{\leq N}\in X\), the result follows from part~(2)
and \(X=\overline{\bigcup_{n\geq0}J_n}\).
\end{proof}

\begin{lemma}\label{lem:length-fresh-direction} The following
hold.
\begin{enumerate}
\item\label{item:fresh-direction-distance}
For every \(n\geq1\), one has
\(\operatorname{dist}(P_n,J_{n-1})=1\).

\item\label{item:fresh-direction-compression}
For every \(n\geq2\), \(T\in P_nJ_{n-1}P_n\), and
\(\varepsilon>0\), there exists a unit vector
\(h\in H_{\bbR}\) such that
\[
 \operatorname{dist}\bigl(
   \ell(h)^*T\ell(h),J_{n-2}
 \bigr)<\varepsilon.
\]
\end{enumerate}
\end{lemma}

\begin{proof}
\emph{(1)} By part~(\ref{item:length-filtration-dense-subalgebra})
of Lemma~\ref{lem:length-filtration} and the density of elementary
tensors, it suffices to consider
\(T\in D_{n-1}\) whose tensor factors lie in the complexification
\(K\) of a finite-dimensional real subspace
\(K_{\bbR}\subset H_{\bbR}\).  Choose a unit vector
\(h\in H_{\bbR}\ominus K_{\bbR}\), and put
\(\zeta_n=h^{\otimes n}/\norm{h^{\otimes n}}_q\).  Then
\(T\zeta_n=0\), whereas \(P_n\zeta_n=\zeta_n\); the conclusion
follows.

\emph{(2)} Choose
\[
 T_0=P_n\left(
 \sum_{j=1}^m\lambda_j
 \ell(\xi_j)r(\eta_j)P_{k_j}
 \ell(\xi'_j)^*r(\eta'_j)^*
 \right)P_n\in P_nD_{n-1}P_n
\]
such that \(0\leq k_j\leq n-1\) and
\(\norm{T-T_0}<(1-|q|)\varepsilon\), where all the displayed
tensors are elementary tensors with factors in \(H_{\bbR}\).  Choose
a finite-dimensional real subspace \(K_{\bbR}\subset H_{\bbR}\)
containing all tensor factors occurring in the displayed expression
for \(T_0\), let \(K\) be its complexification, and choose a unit
vector \(h\in H_{\bbR}\ominus K_{\bbR}\).

For \(1\leq j\leq m\), since \(h\perp K\), the commutation relations show that the compression of the
\(j\)-th summand is
\[
 \lambda_jq^{|\xi_j|+|\xi'_j|}P_{n-1}\ell(\xi_j)r(\eta_j)
 \ell(h)^*\ell(h)P_{k_j-1}
 \ell(\xi'_j)^*r(\eta'_j)^*P_{n-1}.
\]
Since \(k_j-1\leq n-2\), the proof of
part~(\ref{item:length-filtration-dense-subalgebra}) of
Lemma~\ref{lem:length-filtration} shows that
\(\ell(h)^*T_0\ell(h)\in J_{n-2}\).  Since
\(\norm{\ell(h)}^2\leq(1-|q|)^{-1}\), we have
\(\operatorname{dist}(\ell(h)^*T\ell(h),J_{n-2})
\leq\norm{\ell(h)}^2\norm{T-T_0}<\varepsilon\).
\end{proof}

\begin{proof}[Proof of Theorem~\ref{thm:length-filtration-main}]
\emph{(1)} This is part~(\ref{item:length-filtration-increasing}) of
Lemma~\ref{lem:length-filtration}.

\emph{(2)} Fix \(n\geq1\), put \(B=C/J_{n-1}\), and let
\(p_n=P_n+J_{n-1}\).  Part~(\ref{item:fresh-direction-distance}) of
Lemma~\ref{lem:length-fresh-direction} gives \(p_n\ne0\).

To prove that \(p_n\) is minimal, it suffices to show that
\(P_nCP_n\subset\bbC P_n+J_{n-1}\).  Since
\(P_nCP_n=P_nJ_nP_n\) and \(D_n\) is dense in
\(J_n\) by part~(\ref{item:length-filtration-dense-subalgebra}) of
Lemma~\ref{lem:length-filtration}, it
suffices to consider
\[
 T=\ell(\xi)r(\eta)P_k\ell(\xi')^*r(\eta')^*,
 \qquad 0\leq k\leq n.
\]
If \(k<n\), then \(T\in D_{n-1}\subset J_{n-1}\).  If
\(k=n\), then
\(P_nTP_n=0\) unless
\(|\xi|+|\eta|=|\xi'|+|\eta'|=0\), in which case
\(P_nTP_n\in\bbC P_n\).  Therefore
\(P_nCP_n\subset\bbC P_n+J_{n-1}\), and hence
\(p_nBp_n=\bbC p_n\).  Thus \(p_n\) is minimal.

Since \(p_n\in J_n/J_{n-1}\), one has
\(Bp_n=(J_n/J_{n-1})p_n\).  Set
\(E_n=\overline{(J_n/J_{n-1})p_n}\) and equip it with the
Hilbert-space inner product determined for
\(x,y\in J_n/J_{n-1}\) by
\(p_nx^*yp_n=\langle xp_n,yp_n\rangle p_n\).  Under left
multiplication on \(E_n\), for \(x,y,z\in J_n/J_{n-1}\), one has
\((xp_ny^*)(zp_n)=xp_n\langle yp_n,zp_n\rangle\).  Thus \(xp_ny^*\)
acts as the rank-one operator associated with the vectors \(xp_n\) and \(yp_n\).
Since the linear span of such elements is dense in \(J_n/J_{n-1}\),
left multiplication identifies \(J_n/J_{n-1}\) with \(\K(E_n)\).

\emph{(3)} Fix \(n\geq1\), and let
\(0\ne I\lhd X/J_{n-1}\).  Let
\(\pi_n:X\to X/J_{n-1}\) be the quotient map, put
\(\overline I=\pi_n^{-1}(I)\), and then
\(J_{n-1}\subsetneq\overline I\subset X\).
For \(x_0\in\overline I\setminus J_{n-1}\), as
\((P_{\leq N})_N\) is an approximate identity for \(X\) by
part~(\ref{item:length-filtration-approximate-identity}) of
Lemma~\ref{lem:length-filtration}, one has
\(P_{\leq N}x_0\to x_0\), with
\(P_{\leq N}x_0\in\overline I\cap J_N\).  Hence there is a least
\(s\geq n\) such that
\(\overline I\cap J_s\not\subset J_{n-1}\).  We claim that
\(s=n\), from which the result follows.
Indeed, by minimality of \(s\),
\(((\overline I\cap J_s)+J_{s-1})/J_{s-1}\) is a nonzero
ideal of \(J_s/J_{s-1}\), and hence equals
\(J_s/J_{s-1}\) by part~(2).  Thus, once \(s=n\), the
equality \(\pi_n(\overline I)=I\) gives
\(J_n/J_{n-1}\subset I\).

Now suppose \(s>n\) and choose \(x\in\overline I\cap J_s\) such
that \(x-P_s\in J_{s-1}\), and put \(y=P_sx^*xP_s\).  Then
\(y\in\overline I\), \(y\geq0\), \(y=P_syP_s\), and
\(y-P_s\in J_{s-1}\).  Thus \(y=P_s+T\) for some
\(T\in P_sJ_{s-1}P_s\).
Given \(\varepsilon>0\), choose a unit vector
\(h_\varepsilon\in H_{\bbR}\) as in
part~(\ref{item:fresh-direction-compression}) of
Lemma~\ref{lem:length-fresh-direction}, and put
\(z_\varepsilon=\ell(h_\varepsilon)^*y
\ell(h_\varepsilon)\).  Since \(y=P_s+T\), one has
\(z_\varepsilon=\ell(h_\varepsilon)^*P_s
\ell(h_\varepsilon)+\ell(h_\varepsilon)^*T
\ell(h_\varepsilon)\).  One computes that
\[
 \ell(h_\varepsilon)^*P_s\ell(h_\varepsilon)-P_{s-1}
 =q\ell(h_\varepsilon)P_{s-2}\ell(h_\varepsilon)^*
 \in J_{s-2},
\]
whereas part~(\ref{item:fresh-direction-compression}) of
Lemma~\ref{lem:length-fresh-direction} gives
\(\operatorname{dist}(\ell(h_\varepsilon)^*T
\ell(h_\varepsilon),J_{s-2})<\varepsilon\).  Consequently,
\(\operatorname{dist}(z_\varepsilon-P_{s-1},J_{s-2})
<\varepsilon\).

Note that \(P_s\ell(h_\varepsilon)P_{s-1}
=P_ss(h_\varepsilon)P_{s-1}\in X\).  Since \(y=P_syP_s\),
\(z_\varepsilon=(P_s\ell(h_\varepsilon)P_{s-1})^*y
(P_s\ell(h_\varepsilon)P_{s-1})\in\overline I\).

The preceding estimate and the closedness of
\((\overline I+J_{s-2})/J_{s-2}\) imply that there exists
\(z\in\overline I\) such that
\(z+J_{s-2}=P_{s-1}+J_{s-2}\).  It then follows from
part~(\ref{item:fresh-direction-distance}) of
Lemma~\ref{lem:length-fresh-direction} that
\(z\in J_{s-1}\setminus J_{s-2}\), from which the claim
follows, as \(s-2\geq n-1\).
\end{proof}

\subsection{Approximation for minimal projections}
\label{subsec:first-layer-minimal-projections}

The filtration in Section~\ref{subsec:ideal-filtration} identifies
\(J_1\) as the least noncompact ideal of
\(X\). In the proof of the main theorem, this
forces the image of \(P_1\) under an isomorphism into the first layer
of the target algebra, though it need not be the target \(P_1\).
We therefore turn to \(J\)-invariant minimal projections in
\(J_1/\K(\mathcal F_q(H))\), for which
Proposition~\ref{prop:first-layer-real-minimal-projection} provides a
concrete approximation using the Fock-space model.

We first fix some notation.

\begin{notation}\label{not:first-layer-columns}
Recall from Theorem~\ref{thm:length-filtration-main} and its proof that
\(J_0=\K(\F)\). Put \(p_1=P_1+\K(\F)\) and
\(E_1=\overline{(J_1/\K(\F))p_1}\), with inner product determined by
\(p_1x^*yp_1=\langle xp_1,yp_1\rangle p_1\) for
\(x,y\in J_1/\K(\F)\). Its Hilbert-space norm agrees with the quotient
norm, and left multiplication identifies \(J_1/\K(\F)\) with
\(\K(E_1)\). Let \(\mathcal E_{1,\mathrm{alg}}\) be the linear span
of \(\xi\odot\eta\), where \(\xi,\eta\in\F\) are elementary. For
\(\alpha=\sum_{j=1}^m\lambda_j\xi_j\odot\eta_j
\in\mathcal E_{1,\mathrm{alg}}\), where \(\lambda_j\in\bbC\) and
\(\xi_j,\eta_j\in\F\) are elementary, put
\(V_\alpha=\sum_{j=1}^m\lambda_j\ell(\xi_j)r(\eta_j)P_1\). Since
\(V_\alpha=V_\alpha P_1\), we write
\(v_\alpha=V_\alpha+\K(\F)\in E_1\). If
\(K_{\bbR}\subset H_{\bbR}\) is finite-dimensional with complexification
\(K\), we say that \(\alpha\), or \(V_\alpha\), is supported on \(K\) if
\(\alpha\) admits such a representation with
\(\xi_j,\eta_j\in\mathcal F_q(K)\).
\end{notation}

The following is the key approximation technique that will be used to prove the main theorem.

\begin{proposition}
\label{prop:first-layer-real-minimal-projection}
Let \(e\in J_1\) be a projection
such that \(JeJ=e\), and suppose
that \(e+\K(\F)\) is a nonzero minimal projection of
\(J_1/\K(\F)\).

For every \(\varepsilon>0\), there exist a finite-dimensional real
subspace \(K_{\bbR}\subset H_{\bbR}\), an element
\(\alpha=\sum_{j=1}^m\lambda_j\,\xi_j\odot\eta_j\in \mathcal E_{1,\text{alg}}\), where
\(\lambda_j\in\bbC\) and \(\xi_j,\eta_j\in\mathcal F_q(K)\) are
elementary tensors, such that
\[
 JV_\alpha J=V_\alpha,\qquad
 V_\alpha^*V_\alpha-P_1\in\K(\F),\qquad
 \operatorname{dist}(e-V_\alpha V_\alpha^*,\K(\F))<\varepsilon.
\]
Moreover, there exist orthonormal vectors
\(h_1,h_2\in H_{\bbR}\ominus K_{\bbR}\) and orthonormal
vectors \(\psi_1,\psi_2\in e\F\) satisfying \(J\psi_i=\psi_i\)
such that \(\norm{\psi_i-V_\alpha h_i}<\varepsilon\) for
\(i=1,2\).
\end{proposition}

It is helpful
to view the above proposition as an approximated lifting problem. Under the identification
\(J_1/\K(\F)\cong\K(E_1)\subset\B(\F)/\K(\F)\), write
\(e+\K(\F)=\theta_{\chi,\chi}\) for a unit vector \(\chi\in E_1\).
The rank-one operator \(\theta_{\chi,p_1}\) has initial projection
\(p_1\) and final projection \(e+\K(\F)\). Thus, modulo compact
operators, a lift of \(\theta_{\chi,p_1}\) may be viewed as an
isometry from \(P_1\F\) to \(e\F\), thereby connecting the abstract
projection \(e\) with the well-understood projection \(P_1\). Also, we note that although the last statement only provides two vectors $h_1,h_2$, the same proof also works if we want $n$ orthonomal vectors $h_1,\cdots,h_n\in H_{\bbR }\ominus K_{\bbR}$ with the same properties for any $n\geq 1$. 

The following lemma identifies a dense algebraic subspace of \(E_1\)
on which this lifting picture is concrete. For \(\varphi=v_\alpha\),
the operators \(V_\alpha\) and \(V_\alpha V_\alpha^*\) lift
\(\theta_{\varphi,p_1}\) and \(\theta_{\varphi,\varphi}\), respectively,
while \(V_\alpha^*V_\alpha\) records the norm of \(\varphi\) modulo
compact operators.

\begin{lemma}\label{lem:first-layer-columns}
The following hold.
\begin{enumerate}
\item\label{item:first-layer-density}
The map
\[
 \mathcal E_{1,\mathrm{alg}}\longrightarrow E_1,
 \qquad \alpha\longmapsto v_\alpha,
\]
has dense range. Moreover, the formula
\[
 [\alpha,\beta]_1p_1
 =V_\alpha^*V_\beta+\K(\F)
 \qquad(\alpha,\beta\in\mathcal E_{1,\mathrm{alg}})
\]
defines a positive semidefinite sesquilinear form on
\(\mathcal E_{1,\mathrm{alg}}\).
The Hilbert space obtained from this form by separation and completion
is isomorphic to \(E_1\) through the map above.

\item\label{item:first-layer-scalar-mod-compacts}
For \(\alpha\in\mathcal E_{1,\mathrm{alg}}\), put
\(c=[\alpha,\alpha]_1=\norm{v_\alpha}_{E_1}^2\). Then
\(V_\alpha^*V_\alpha-cP_1\in\K(\F)\). If \(\alpha\) and \(\beta\)
are supported on \(K\), then
\[
 \langle V_\alpha h,V_\beta g\rangle
 =[\alpha,\beta]_1\langle h,g\rangle
 \qquad(h,g\in K^\perp).
\]
\end{enumerate}
\end{lemma}

\begin{proof}
Since \(v_\alpha=v_\alpha p_1\), the definition of the inner product
on \(E_1\) gives
\[
 V_\alpha^*V_\beta+\K(\F)
 =\langle v_\alpha,v_\beta\rangle_{E_1}p_1.
\]
Thus \([\alpha,\beta]_1=\langle v_\alpha,v_\beta\rangle_{E_1}\),
which proves the assertion about the sesquilinear form.

By the density of \(D_1\) in \(J_1\) and
the norm identity in Notation~\ref{not:first-layer-columns}, it
suffices to show that for every \(d\in D_1\), the vector
\((d+\K(\F))p_1\) has the form \(v_\alpha\).  By linearity,
consider a standard term
\(d=\ell(\xi)r(\eta)P_k\ell(\xi')^*r(\eta')^*\), where \(k=0\) or
\(1\).  Note that the only nonzero case with \(k=1\) gives
\(dP_1=V_{\xi\odot\eta}\), while \(dP_1\in\K(\F)\) when \(k=0\).
This proves density, and the completion statement follows.

Taking \(\beta=\alpha\) in the first displayed identity, with
\(c=[\alpha,\alpha]_1\), gives
\(V_\alpha^*V_\alpha-cP_1\in\K(\F)\). Now suppose that \(\alpha\)
and \(\beta\) are supported on \(K\). For \(h,g\in K^\perp\), the
\(q\)-Fock inner-product formula shows that every nonzero term in
\(\langle V_\alpha h,V_\beta g\rangle\) pairs \(h\) with \(g\).
Hence there is a scalar \(d\), depending only on \(\alpha\) and
\(\beta\), such that
\[
 \langle V_\alpha h,V_\beta g\rangle
 =d\langle h,g\rangle
 \qquad(h,g\in K^\perp).
\]
Let \((h_n)\) be an orthonormal sequence in \(K^\perp\). Since
\(V_\alpha^*V_\beta-[\alpha,\beta]_1P_1\) is compact,
\[
 d=\langle V_\alpha h_n,V_\beta h_n\rangle
  =[\alpha,\beta]_1
   +\langle h_n,
      (V_\alpha^*V_\beta-[\alpha,\beta]_1P_1)h_n\rangle
  \longrightarrow[\alpha,\beta]_1.
\]
Thus \(d=[\alpha,\beta]_1\), as required.
\end{proof}

For a general minimal projection
\(e+\K(\F)=\theta_{\chi,\chi}\), there is no comparable concrete
Fock-space formula for a lift of \(\theta_{\chi,p_1}\). We therefore
approximate \(\chi\) by algebraic vectors \(v_\alpha\). The operator $V_\alpha$ can then be considered as an approximate essential isometry from $P_1$ to $e$. This leads to the proof of Proposition~\ref{prop:first-layer-real-minimal-projection}. 

\begin{proof}[Proof of Proposition~\ref{prop:first-layer-real-minimal-projection}]
Conjugation by \(J\) preserves \(J_1\), \(\K(\F)\), and \(p_1\), and
hence induces a conjugation on \(E_1\). For elementary tensors
\(\xi,\eta\), one has
\(JV_{\xi\odot\eta}J=V_{J\eta\odot J\xi}\), so the image of
\(\mathcal E_{1,\mathrm{alg}}\) in \(E_1\) is invariant under this
conjugation.

Under \(J_1/\K(\F)\cong\K(E_1)\), write
\(e+\K(\F)=\theta_{\chi,\chi}\) for a unit vector \(\chi\in E_1\),
where \(\theta_{\xi,\eta}(\zeta)=\xi\langle\eta,\zeta\rangle\).
Since \(JeJ=e\), we may assume that \(\chi\) is invariant under the
conjugation.

Fix \(\delta>0\), to be chosen sufficiently small later. By
part~(\ref{item:first-layer-density}) of
Lemma~\ref{lem:first-layer-columns}, after averaging an algebraic
approximation to \(\chi\) with its conjugate and normalizing, we may
choose \(\alpha\in\mathcal E_{1,\mathrm{alg}}\) such that, with
\(\varphi=v_\alpha\),
\(
 JV_\alpha J=V_\alpha,\;
 \norm{\varphi}_{E_1}=1,\; \) and \(
 \norm{\varphi-\chi}_{E_1}<\delta.
\)
Then
part~(\ref{item:first-layer-scalar-mod-compacts}) of
Lemma~\ref{lem:first-layer-columns} gives
\(V_\alpha^*V_\alpha-P_1\in\K(\F)\).

Under the identification \(J_1/\K(\F)\cong\K(E_1)\), the class
\(V_\alpha+\K(\F)\) is \(\theta_{\varphi,p_1}\). Therefore
\(V_\alpha V_\alpha^*+\K(\F)=\theta_{\varphi,\varphi}\), and hence
\[
 \operatorname{dist}(e-V_\alpha V_\alpha^*,\K(\F))
 =\norm{\theta_{\chi,\chi}-\theta_{\varphi,\varphi}}
 \leq2\norm{\chi-\varphi}_{E_1}<2\delta.
\]
For the moreover part, note that
\[
 \norm{(1-e)V_\alpha+\K(\F)}
 =\norm{(1-\theta_{\chi,\chi})\theta_{\varphi,p_1}}
 =\norm{(1-\theta_{\chi,\chi})\varphi}_{E_1}
 \leq\norm{\varphi-\chi}_{E_1}<\delta.
\]
We may then choose \(T\in\K(\F)\) such that
\(\norm{(1-e)V_\alpha+T}<2\delta\).
Choose a finite-dimensional real subspace
\(K_{\bbR}\subset H_{\bbR}\) whose complexification \(K\) supports
\(\alpha\).  Since \(T\) is compact, we may find orthonormal vectors
\(h_1,h_2\in H_{\bbR}\ominus K_{\bbR}\) such that
\(\norm{(1-e)V_\alpha h_i}<3\delta\) for \(i=1,2\).
Put \(\psi_i'=V_\alpha h_i\). These vectors are \(J\)-real, and by
part~(\ref{item:first-layer-scalar-mod-compacts}) of
Lemma~\ref{lem:first-layer-columns}, the vectors
\(\psi_1',\psi_2'\) are orthonormal.

We can now correct $ e\psi_1',e\psi_2'$ by an invertible Gram matrix to get the desired $\psi_1,\psi_2$. Let \(U_0:\bbC^2\to\F\) be the isometry with columns
\(\psi_1',\psi_2'\), and put \(W=eU_0\). Then
\(\norm{W-U_0}\leq3\sqrt2\,\delta\), so \(W^*W\to I_2\) as
\(\delta\to0\). For all sufficiently small \(\delta\), define
\(
 \Psi=W(W^*W)^{-1/2}:\bbC^2\longrightarrow e\F.
\)
Then \(\Psi\) is an isometry and \(\Psi\to U_0\) as \(\delta\to0\).
Let \(\psi_1,\psi_2\) be its columns. Since the columns of \(W\) are
\(J\)-real, \(W^*W\) is real symmetric, and hence
\(J\psi_i=\psi_i\).
Choosing \(\delta\) sufficiently small gives simultaneously
\(\operatorname{dist}(e-V_\alpha V_\alpha^*,\K(\F))
<\varepsilon\) and
\(\norm{\psi_i-V_\alpha h_i}<\varepsilon\) for \(i=1,2\).
\end{proof}

\subsection{First Chaos Projection}
\label{sec:first-chaos-projections}

We now summarize some abstract properties of the projection $P_1$ which will later be enough to recover the parameter $q$. 

\begin{definition}
\label{def:first-chaos-projection}
A projection
\(e\in J_1\) is called a \emph{first chaos projection} if
\begin{enumerate}
\item \(JeJ=e\), \(e\Omega=0\), and \(e\) is \(M\)-regular;
\item \(e+\K(\F)\) is a nonzero minimal projection of
      \(J_1/\K(\F)\);
\item The map
\[
 e\F\odot e\F\ni x\Omega\otimes y\Omega
 \longmapsto
 \bigl(xy-\tau(xy)1\bigr)\Omega\in\F\ominus\bbC\Omega
\]
extends to a bounded operator \(Q_e\) on \(e\F\otimes e\F\), and there exists
\(r_e\in(-1,1)\) such that \(Q_e^*Q_e=I+r_eF_e\), where
\(F_e(\xi\otimes\eta)=\eta\otimes\xi\) is the tensor flip on
\(e\F\otimes e\F\).
\end{enumerate}
\end{definition}

The terminology is motivated by the canonical projection \(P_1\).
For \(h,g\in H\), one has
\[
 s(h)s(g)\Omega
 =h\otimes g+\langle Jh,g\rangle\Omega.
\]
Consequently, \(Q_{P_1}(h\otimes g)=h\otimes g\), where the domain has
the usual tensor-product inner product and the range has the
two-particle \(q\)-Fock inner product.  It follows that
\(
 Q_{P_1}^*Q_{P_1}=I+qF_{P_1},
\) hence $P_1$ is a first chaos projection.

\begin{proposition}\label{prop:first-chaos-parameter}
Let \(e\in J_1\) be a
first chaos projection and \(r_e\in(-1,1)\) satisfy
\(Q_e^*Q_e=I+r_eF_e\). Then \(r_e\in q[-q^2,1]\).
\end{proposition}

The final step of the following proof uses the estimate established in
Proposition~\ref{prop:ordered-swap-estimate}.

\begin{proof}
Fix \(\varepsilon>0\).  Since \(e\in J_1\), \(JeJ=e\), and
\(e+\K(\F)\) is a nonzero minimal projection in
\(J_1/\K(\F)\), Proposition~\ref{prop:first-layer-real-minimal-projection}
gives a finite-dimensional real
subspace \(K_{\bbR}\subset H_{\bbR}\), an \(\alpha\in \mathcal E_{1,\text{alg}}\)
supported on the complexification \(K\), orthonormal vectors
\(h_1,h_2\in H_{\bbR}\ominus K_{\bbR}\), and orthonormal
\(J\)-real vectors \(\psi_1,\psi_2\in e\F\) such that
\(\norm{\psi_i-V_\alpha h_i}<\varepsilon\) for \(i=1,2\).

Let \(U\in\mathcal O(H_{\bbR})\) exchange \(h_1\) and \(h_2\)
and fix \(\{h_1,h_2\}^{\perp}\) pointwise, and let \(S=\mathcal F_q(U)\) be the second quantization operator
\[ S=\mathcal F_q(U): \F \to \F,\quad \xi_1\otimes \cdots \otimes \xi_n\mapsto (U\xi_1)\otimes \cdots \otimes (U\xi_n). \]  Since \(\alpha\) is supported on
\(K\subset\{h_1,h_2\}^{\perp}\), we have
\(S V_\alpha h_1=V_\alpha h_2\) and
\(S V_\alpha h_2=V_\alpha h_1\).

Since \(e\) is \(M\)-regular, Lemma~\ref{lem:regular-corners} gives
a constant \(c<\infty\) and elements \(x_i\in M\) such that
\(\psi_i=x_i\Omega\) and \(\norm{x_i}_\infty\leq c\). Moreover,
\(x_i\) is self-adjoint, as
\(J\psi_i=\psi_i\).  It follows that
\(\tau(x_1x_2)=\tau(x_2x_1)=0\), and hence
\(Q_e(\psi_1\otimes\psi_2)=x_1x_2\Omega\) and
\(Q_e(\psi_2\otimes\psi_1)=x_2x_1\Omega\).

We claim that
\begin{equation}
 \norm{x_1x_2}_2=\norm{x_2x_1}_2=1,
 \qquad
 \langle x_1x_2\Omega,x_2x_1\Omega\rangle=r_e.
 \label{eq:first-chaos-product-correlation}
\end{equation}
Indeed,
\[
 \norm{x_1x_2}_2^2
 =\big\langle Q_e^*Q_e(\psi_1\otimes\psi_2),
                  \psi_1\otimes\psi_2\big\rangle
 =\big\langle (I+r_eF_e)(\psi_1\otimes\psi_2),
                  \psi_1\otimes\psi_2\big\rangle=1,
\]
and similarly \(\norm{x_2x_1}_2=1\).  We also compute
\[
 \langle x_1x_2\Omega,x_2x_1\Omega\rangle
 =\big\langle Q_e^*Q_e(\psi_1\otimes\psi_2),
                  \psi_2\otimes\psi_1\big\rangle
 =\big\langle (I+r_eF_e)(\psi_1\otimes\psi_2),
                  \psi_2\otimes\psi_1\big\rangle=r_e.
\]
Here the inner products are taken in $ \F \otimes \F $.

Note that
\(\norm{Sx_1\Omega-x_2\Omega},\allowbreak\
\norm{Sx_2\Omega-x_1\Omega}<2\varepsilon\), as
\(\norm{\psi_i-V_\alpha h_i}<\varepsilon\).  Since \(S\)
implements an automorphism of \(M\), it follows that
\begin{equation}
 \norm{Sx_1x_2\Omega-x_2x_1\Omega}
 \leq c\norm{Sx_1\Omega-x_2\Omega}
      +c\norm{Sx_2\Omega-x_1\Omega}<4c\varepsilon.
 \label{eq:first-chaos-swapped-products}
\end{equation}

Put \(L_i=K_{\bbR}\oplus\bbR h_i\), let
\(E_i:M\to M_q(L_i)\) be the conditional expectation, and set
\(y_i=E_i(x_i)\).  We note that
\begin{equation*}
 \norm{y_i\Omega-\psi_i},\
 \norm{y_i\Omega-V_\alpha h_i}<\varepsilon,
\end{equation*}
as \(V_\alpha h_i\in\mathcal F_q(L_i)\), while
\(\norm{V_\alpha h_i-\psi_i}<\varepsilon\), and we have
\(\norm{y_i}_\infty\leq\norm{x_i}_\infty\leq c\).  It follows that
\(\norm{x_i-y_i}_2<\varepsilon\), and hence
\begin{equation}
 \norm{x_1x_2-y_1y_2}_2,\
 \norm{x_2x_1-y_2y_1}_2<2c\varepsilon.
 \label{eq:first-chaos-product-approximations}
\end{equation}

For \(i=1,2\), let \(R_i\) be the orthogonal projection onto the
closed span of the elementary tensors in $ \{h_1,h_2\}\cup\{h_1,h_2\}^{\perp} $ containing exactly one
\(h_i\)-factor.  Since \(V_\alpha h_i\) is in the range of
\(R_i\), \(\norm{y_i\Omega-V_\alpha h_i}<\varepsilon\)
gives \(\norm{(1-R_i)y_i\Omega}<\varepsilon\).  Observe that \(R_1\)
commutes with the right action of \(y_2\), while \(R_2\) commutes with
the left action of \(y_1\), and hence
\(\norm{(1-R_1)y_1y_2\Omega},\
\norm{(1-R_2)y_1y_2\Omega}\leq c\varepsilon\).

Set \(\xi_\varepsilon=R_1R_2y_1y_2\Omega\).  The preceding estimates
give
\begin{equation}
 \norm{\xi_\varepsilon-y_1y_2\Omega}
 \leq\norm{(1-R_1)y_1y_2\Omega}
      +\norm{(1-R_2)y_1y_2\Omega}
 \leq2c\varepsilon.
 \label{eq:first-chaos-counting-approximation}
\end{equation}

Let \(K_+\subset\mathcal F_q(H)\) be the closed span of the
elementary tensors containing exactly one \(h_1\)-factor and one
\(h_2\)-factor, with all remaining factors in
\(\{h_1,h_2\}^{\perp}\), and with the \(h_1\)-factor
occurring first.  We have \(\xi_\varepsilon\in\mathcal K_+\).  Indeed,
\(R_1R_2\) selects the terms containing exactly one
\(h_1\)-factor and one \(h_2\)-factor.  Since these factors
come from \(y_1\) and \(y_2\), respectively, their order in the
product \(y_1y_2\) places \(h_1\) first.  The assertion follows
first for Wick polynomials and then by \(L^2\)-approximation.

Combining equations \eqref{eq:first-chaos-product-approximations} and
\eqref{eq:first-chaos-counting-approximation}, and then using
equation \eqref{eq:first-chaos-swapped-products} and the unitarity of
\(S\), we obtain
\(\norm{\xi_\varepsilon-x_1x_2\Omega}\leq4c\varepsilon\) and
\(\norm{S\xi_\varepsilon-x_2x_1\Omega}\leq8c\varepsilon\).
Together with equation \eqref{eq:first-chaos-product-correlation}, these
estimates imply \(\norm{\xi_\varepsilon}\to1\) and
\(\langle \xi_\varepsilon,S\xi_\varepsilon\rangle\to r_e\) as
\(\varepsilon\to0\).

In particular, \(\xi_\varepsilon\neq0\) for sufficiently small
\(\varepsilon\). Since, $ \xi_{\varepsilon}\in K_+ $, Proposition~\ref{prop:ordered-swap-estimate} gives
\(\langle \xi_\varepsilon,S\xi_\varepsilon\rangle/
\norm{\xi_\varepsilon}^2\allowbreak\in q[-q^2,1]\), and it follows that
\(r_e\in q[-q^2,1]\) by taking \(\varepsilon\to0\).
\end{proof}

\section{Proof of the main theorem}

An isomorphism between two \(q\)-Gaussian von Neumann algebras
identifies their corresponding Akemann--Ostrand kernels, as computed in
Section~\ref{sec:akemann-ostrand-kernel}. The structural analysis of
Section~\ref{sec:ideal-filtration-first-layer}, however, takes place in
\(X/\K(\F)\). We therefore begin by relating these two settings in
Proposition~\ref{prop:ao-first-word-length-layer} using $M$-regular projections (Lemma~\ref{lem:regular-corners}).

\begin{lemma}
\label{lem:word-length-bridge}
We have the following:
\begin{enumerate}
\item\label{item:word-length-regularity}
For every \(n\geq0\), the projection \(P_n\) is \(M\)-regular.
\item\label{item:word-length-multipliers}
One has \(M,JMJ\subset\mathcal M(\bbX)\).
\item\label{item:word-length-ideals}
For every \(n\geq0\), one has
\(\overline{\C P_n\C}^{\,\norm{\cdot}}
=\overline{CP_nC}^{\,\norm{\cdot}}=J_n\).
\item\label{item:word-length-corners}
For every \(n\geq0\), one has
\(P_n\C P_n=P_nCP_n\).
\end{enumerate}
\end{lemma}

\begin{proof}
\hyperref[item:word-length-regularity]
{(\ref*{item:word-length-regularity})}
By the Haagerup inequality for $q$-Gaussian algebras, for every
\(n\geq0\) there is
\(d_{q,n}<\infty\) such that
\[
 W(P_n\xi)\Omega=P_n\xi,\qquad
 \norm{W(P_n\xi)}_\infty
 \leq d_{q,n}\norm{P_n\xi}_2
 \leq d_{q,n}\norm{\xi}_2
\]
for every \(\xi\in L^2M\); see
\cite[Proposition~2.1(b)]{Boz99}.  Hence
\(P_n(L^2M)\subset\widehat M\), so \(P_n\) is \(M\)-regular.

\hyperref[item:word-length-multipliers]
{(\ref*{item:word-length-multipliers})}
Note that \(A,JAJ\subset\mathcal M(\bbX)\).  It then follows from
part~(\ref{item:word-length-regularity}) and
Lemma~\ref{lem:regular-corners} that \(xP_n,P_nx\in\bbX\) for
every \(x\in M\cup JMJ\), and hence
\(M,JMJ\subset\mathcal M(\bbX)\).

\hyperref[item:word-length-ideals]
{(\ref*{item:word-length-ideals})}
It suffices to see that
\(x_1Jy_1JP_nJy_2Jx_2\in\overline{CP_nC}^{\,\norm{\cdot}}\) for
\(x_i,y_i\in M\), which follows directly from
part~(\ref{item:word-length-regularity}) and
Lemma~\ref{lem:regular-corners}.  The reverse inclusion follows from
\(C\subset\C\).

\hyperref[item:word-length-corners]
{(\ref*{item:word-length-corners})}
It suffices to see that \(P_nxJyJP_n\in P_nCP_n\) for
\(x,y\in M\), which follows directly from
part~(\ref{item:word-length-regularity}) and the second continuity
assertion of Lemma~\ref{lem:regular-corners}.  The reverse inclusion
is immediate.
\end{proof}

\begin{proposition}
\label{prop:ao-first-word-length-layer}
The projection
\(P_1+\mathcal K\) is a nonzero minimal projection of
\(\mathcal X/\mathcal K\).

Moreover, if \(e\in\mathcal X\) is an \(M\)-regular projection such that
\(e+\mathcal K\) is a nonzero minimal projection of
\(\mathcal X/\mathcal K\), then \(e\in J_1\), and
\(e+\K(\F)\) is a minimal projection of
\(J_1/\K(\F)\).
\end{proposition}

\begin{proof}
Since \(P_1\in\C\cap\bbX\subset\mathcal X\),
part~(\ref{item:word-length-corners}) of
Lemma~\ref{lem:word-length-bridge} and
part~(\ref{item:length-filtration-minimal-projections}) of
Theorem~\ref{thm:length-filtration-main} give
\[
 P_1\C P_1=P_1CP_1
 \subset\bbC P_1+\K(\F)
 \subset\bbC P_1+\mathcal K.
\]
Thus \(P_1+\mathcal K\) is minimal in \(\mathcal X/\mathcal K\).
To see that it is nonzero, choose an orthonormal sequence
\((\xi_i)\subset H_{\bbR}\).  By
\cite[Proposition~3.8]{DKEP22}, one has
\(\norm{TW(\xi_i)\Omega}_2\to0\) for every \(T\in\mathcal K\), while
\(P_1W(\xi_i)\Omega=\xi_i\).

Now let \(e\) satisfy the hypotheses of the second assertion.  We first
show that \(e+\K(\F)\) is a nonzero minimal projection in
\(\C/\K(\F)\).  For \(x\in\C\), one has
\(exe\in\mathcal X\) because
\(\mathcal X\lhd\C\) by
Theorem~\ref{thm:akemann-ostrand-kernel}.
Minimality of \(e+\mathcal K\) in
\(\mathcal X/\mathcal K\) therefore gives \(\lambda_x\in\bbC\) such that
\(z=exe-\lambda_xe\in\mathcal K\).  Choose \(K_i\in\K(\F)\) with
\(\norm{K_i-z}_{\infty,1}\to0\).  Since \(z=eze\), the second
continuity assertion of Lemma~\ref{lem:regular-corners} then shows
that \(z\in\K(\F)\), from which it follows that \(e+\K(\F)\) is a
nonzero minimal projection in \(\C/\K(\F)\).

Next, we claim that \(e\in X\).  Choose \(T_i\in\bbX\) such that
\(\norm{T_i-e}_{\infty,1}\to0\).  Then \(eT_ie\in\bbX\) by
part~(\ref{item:word-length-multipliers}) of
Lemma~\ref{lem:word-length-bridge}, and the second continuity
assertion of Lemma~\ref{lem:regular-corners} gives \(eT_ie\to e\) in
operator norm, and thus \(e\in\bbX\).  By
Notation~\ref{not:q-gaussian-notation},
\((P_{\leq N})_N\) is an approximate identity for \(\bbX\), so
\(P_{\leq N}eP_{\leq N}\to e\) in operator norm.
Part~(\ref{item:length-filtration-chain}) of
Theorem~\ref{thm:length-filtration-main}, together with
Notation~\ref{not:q-gaussian-notation}, gives
\(P_{\leq N}\in J_N\), while
part~(\ref{item:word-length-ideals}) of
Lemma~\ref{lem:word-length-bridge} gives \(J_N\lhd\C\).  Hence
\(P_{\leq N}eP_{\leq N}\in J_N\), and therefore
\(e\in\overline{\bigcup_NJ_N}^{\,\norm{\cdot}}=X\).
Part~(\ref{item:word-length-ideals}) of
Lemma~\ref{lem:word-length-bridge}, together with
part~(\ref{item:length-filtration-chain}) of
Theorem~\ref{thm:length-filtration-main}, also shows that
\(X\lhd\C\).

Let \(L\) be the ideal of \(\C/\K(\F)\) generated by
\(e+\K(\F)\).  Since \(e\in X\lhd\C\), the ideal \(L\) is
nonzero and contained in \(X/\K(\F)\), so
part~(\ref{item:length-filtration-least-ideal}) of
Theorem~\ref{thm:length-filtration-main} gives
\(J_1/\K(\F)\subset L\).  Since \(L\) is generated by a
minimal projection, it is simple.  Therefore \(L=J_1/\K(\F)\), and
hence \(e\in J_1\) and \(e+\K(\F)\) is minimal in
\(J_1/\K(\F)\).
\end{proof}

Before proving the \hyperlink{thm:main}{main theorem}, we first reduce its proof to the
separable case.

\begin{lemma}
\label{lem:separable-reduction}
Let \(H_{\bbR}\) and \(K_{\bbR}\) be infinite-dimensional real
Hilbert spaces and \(q,r\in(-1,1)\), where \(H_{\bbR}\) is
nonseparable and \(K_{\bbR}\) is separable. If \(M_q(H_{\bbR})\) is
isomorphic to \(M_r(H_{\bbR})\), then \(M_q(K_{\bbR})\) is
isomorphic to \(M_r(K_{\bbR})\).
\end{lemma}

\begin{proof}
Let \(\theta:M_q(H_{\bbR})\to M_r(H_{\bbR})\) be an isomorphism.
It suffices to find increasing sequences \((K_n)_{n\geq0}\) and
\((L_n)_{n\geq0}\) of separable infinite-dimensional closed real
subspaces of \(H_{\bbR}\) such that
\[
 \theta(M_q(K_n))\subset M_r(L_n),
 \quad\text{and}\quad
 \theta^{-1}(M_r(L_n))\subset M_q(K_{n+1})
\]
for every \(n\geq0\).

Indeed, put \(K_\infty=\overline{\bigcup_nK_n}\) and
\(L_\infty=\overline{\bigcup_nL_n}\). It follows that
\[
\begin{aligned}
 \theta(M_q(K_\infty))
 &=\left(\bigcup_n\theta(M_q(K_n))\right)''
  \subset\left(\bigcup_nM_r(L_n)\right)''
  =M_r(L_\infty) \\
 &\subset\left(\bigcup_n\theta(M_q(K_{n+1}))\right)''
  =\theta(M_q(K_\infty)).
\end{aligned}
\]

It remains to construct these sequences. We first claim that, for every
separable von Neumann subalgebra \(N\subset M_q(H_{\bbR})\), there
exists a separable closed real subspace \(L_{\bbR}\subset H_{\bbR}\)
such that \(N\subset M_q(L_{\bbR})\). Indeed, first observe that for
every \(\xi\in\mathcal F_q(H)\), there exists a separable closed real
subspace \(L_{\xi,\bbR}\subset H_{\bbR}\) such that
\(\xi\in\mathcal F_q(L_\xi)\). Since \(L^2(N)\) is separable, it
follows that there exists a separable closed real subspace
\(L_{\bbR}\subset H_{\bbR}\) such that
\(L^2(N)\subset\mathcal F_q(L)\). Denote by
\(E_L:M_q(H_{\bbR})\to M_q(L_{\bbR})\) the canonical conditional
expectation. It follows that \(E_L|_N=\operatorname{id}_N\), from which
the claim follows.

Choose separable infinite-dimensional closed real subspaces
\(K_0,L_{-1}\subset H_{\bbR}\). We construct \(K_n\) and \(L_n\)
consecutively. Suppose that \(K_n\) and \(L_{n-1}\) have been
constructed. Since \(\theta(M_q(K_n))\) is separable, there exists a
separable closed real subspace \(L_n'\subset H_{\bbR}\) such
that \(\theta(M_q(K_n))\subset M_r(L_n')\). Put
\(L_n=\overline{\operatorname{span}_{\bbR}
(L_{n-1}\cup L_n')}\). Likewise, since
\(\theta^{-1}(M_r(L_n))\) is separable, there exists a separable closed
real subspace \(K_{n+1}'\subset H_{\bbR}\) such that
\(\theta^{-1}(M_r(L_n))\subset M_q(K_{n+1}')\). Put
\(K_{n+1}=\overline{\operatorname{span}_{\bbR}
(K_n\cup K_{n+1}')}\). This completes the construction, and the
conclusion follows since \(K_\infty\), \(L_\infty\), and
\(K_{\bbR}\) are all separable infinite-dimensional real Hilbert
spaces.
\end{proof}

\begin{proof}[Proof of the Main Theorem]
The implication from $q=r$ is immediate. By
\cite{Cas23} and Lemma~\ref{lem:separable-reduction}, it
suffices to assume that $H_{\bbR}$ is separable and infinite-dimensional
and that $q,r\in(-1,1)\setminus\{0\}$.

Let $\theta:M_q\to M_r$ be an isomorphism. Since $M_q$ and $M_r$ are
factors, $\theta$ preserves their traces, so the standard unitary
$W_\theta:L^2(M_q)\to L^2(M_r)$, given by
$W_\theta(x\Omega_q)=\theta(x)\Omega_r$, intertwines the canonical
conjugations. Hence $\operatorname{Ad}(W_\theta)$ sends $\C_q$ onto
$\C_r$, preserves the $\|\cdot\|_{\infty,1}$-norm, and sends
$\mathcal K_q$ onto $\mathcal K_r$. Since it also intertwines the
multiplication maps, Theorem~\ref{thm:akemann-ostrand-kernel} gives
$\operatorname{Ad}(W_\theta)(\mathcal X_q)=\mathcal X_r$.

Set $e=W_\theta P_1^{(q)}W_\theta^\ast$. By
Proposition~\ref{prop:ao-first-word-length-layer}, $e+\mathcal K_r$ is a
nonzero minimal projection of $\mathcal X_r/\mathcal K_r$. Moreover,
$e$ is $M_r$-regular by
part~(\ref{item:word-length-regularity}) of
Lemma~\ref{lem:word-length-bridge}. The second assertion of
Proposition~\ref{prop:ao-first-word-length-layer}, together with the fact
that $W_\theta$ intertwines the conjugations, vacuum vectors, and
multiplication, shows that $e$ is a first chaos projection for $M_r$
satisfying $Q_e^\ast Q_e=I+qF_e$.

Proposition~\ref{prop:first-chaos-parameter} therefore gives
$q\in r[-r^2,1]$. Applying the same argument to $\theta^{-1}$ gives
$r\in q[-q^2,1]$. If $q,r>0$, then $r\leq q\leq r$, while if
$q,r<0$, the analogous inequalities again give $q=r$.

Finally, their signs cannot be opposite. If $q>0>r$, then the two
inclusions give $|r|\leq q^3$ and $q\leq|r|^3$, so
$q\leq|r|^3\leq q^9<q$, a contradiction. The case $r>0>q$ is
symmetric. Therefore $q=r$.
\end{proof}

\begin{remark}
The argument above appears adaptable to the stable-isomorphism setting.
Nevertheless, we do not pursue this extension here, as it would require
substantial additional notation and bookkeeping.
\end{remark}

Combining the \hyperlink{thm:main}{main theorem} with the unique trace property of $A_q$
\cite[Theorem~4.1]{BCKW22}, we have the following.

\begin{corollary}
Let $H_{\bbR}$ be an infinite-dimensional real Hilbert space and
$q,r\in(-1,1)$.
Then the $q$-Gaussian $C^*$-algebra $A_q(H_{\bbR})$ is isomorphic to
$A_r(H_{\bbR})$ if and only if $q=r$.
\end{corollary}

\appendix
\section{An ordered swap estimate on \(q\)-Fock space}
\label{sec:ordered-swap-estimate}

Let \(H_{\bbR}\) be a real Hilbert space and set
\(H=H_\bbR\oplus iH_\bbR\). Fix \(q\in(-1,1)\) and orthonormal
vectors \(h_1,h_2\in H_{\bbR}\). Let \(U\in\mathcal O(H_{\bbR})\)
interchange \(h_1\) and \(h_2\), fix \(\{h_1,h_2\}^\perp\)
pointwise, and let \(S=\mathcal F_q(U)\) be the second quantization operator
\[ S=\mathcal F_q(U): \F \to \F,\quad \xi_1\otimes \cdots \otimes \xi_n\mapsto (U\xi_1)\otimes \cdots \otimes (U\xi_n). \]

The motivation for considering this swap is the elementary identity 
\[
\langle S(h_1\otimes h_2), (h_1\otimes h_2)\rangle_q=q
\]
in $q$-Fock space. 
The purpose of this appendix is to show that this dependence on $q$ persists 
in the following general case.

\begin{proposition}
\label{prop:ordered-swap-estimate}
For every nonzero \(\xi\in K_+\),
one has \(\langle S\xi,\xi\rangle_q/\norm{\xi}_q^2\in q[-q^2,1]\),
where 
\[
 K_+
 =\overline{\operatorname{span}}^{\,\norm{\cdot}_q}
 \left\{
 \eta_1\otimes\cdots\otimes\eta_n
 \ \middle|\
 \begin{aligned}
 &n\geq2,\ \text{and for some }1\leq i<j\leq n,\\
 &\eta_i=h_1,\qquad \eta_j=h_2,\\
 &\eta_k\in \{ h_1, h_2\}^\perp\quad\text{for every }k\notin\{i,j\}
 \end{aligned}
 \right\}
 \subset\F.
\]
\end{proposition}

In \cite{MS03}, Meljanac and Svrtan developed a useful combinatorial framework for studying Gram matrices in deformed Fock spaces, which, among other applications, recovers the positivity of the (q)-Fock inner product proved in \cite{BoSp94}. Our proof essentially uses this framework, particularly their Gram-matrix formula for distinct letters and their reduction of repeated letters to labelled copies \cite[Proposition~1.6.1 and Section~1.7]{MS03}. We give a self-contained argument adapted to our setting.

We first treat the case in which all letters are distinct.

\begin{lemma}
\label{lem:distinct-ordered-swap}
Let \(F=\{f_1,\ldots,f_m\}\subset\{h_1,h_2\}^{\perp}\) be a finite
orthonormal set. 
For every nonzero \(\xi\in K_+[F]\), one has
\(\langle S\xi,\xi\rangle_q/\norm{\xi}_q^2\in q[-q^2,1]\), where
\(K_+[F]\) is the span of the elementary tensors in which
\(h_1,h_2,f_1,\ldots,f_m\) occur exactly once, with \(h_1\) preceding
\(h_2\).
\end{lemma}

\begin{proof}
Let \(\mathcal W\) be the set of all words formed from
\(F\cup\{h_1,h_2\}\), with every vector occurring exactly once, and
let \(\mathcal W_+\subset\mathcal W\) consist of those words in which
\(h_1\) occurs before \(h_2\). Let $n=m+2$, for
\(w=(g_1,\ldots,g_n)\in\mathcal W\), write
\(\xi_w=g_1\otimes\cdots\otimes g_n\).  If
\(v=(g_{\sigma(1)},\ldots,g_{\sigma(n)})\) for
\(\sigma\in S_n\), set \(d(w,v)=\inv(\sigma)\).  Equivalently,
\(d(w,v)\) is the number of pairs of vectors whose relative orders in
\(w\) and \(v\) are opposite.

Since the vectors are distinct and orthonormal, exactly one term in
the defining sum for the \(q\)-inner product is nonzero.  Therefore
\(\langle\xi_w,\xi_v\rangle_q=q^{d(w,v)}\). This formula suggests recording the relative order of each pair in a
separate two-dimensional factor as following.

Fix a total order \(\prec\) on \(F\cup\{h_1,h_2\}\) such that
\(f\prec h_1\prec h_2\) for every \(f\in F\), and let
\[
 \mathcal P
 =\{(g,g'):g,g'\in F\cup\{h_1,h_2\},\ g\prec g'\}.
\]
For \(p=(g,g')\in\mathcal P\), define \(\varepsilon_p(w)\) to be
\(0\) if \(g\) occurs before \(g'\) in \(w\), and \(1\) otherwise.

Let \(D=\bbC^2\) with distinguished basis \(d_0,d_1\), equipped
with the inner product determined by \(\langle d_i,d_i\rangle=1\) for
\(i=0,1\) and
\(\langle d_0,d_1\rangle=\langle d_1,d_0\rangle=q\).  Using one copy
\(D_p\) of \(D\) for every \(p\in\mathcal P\), define
\[
 V(\xi_w)
 =\bigotimes_{p\in\mathcal P}d_{\varepsilon_p(w)}
 \in\bigotimes_{p\in\mathcal P}D_p.
\]
The number of coordinates in which \(V(\xi_w)\) and \(V(\xi_v)\)
differ is \(d(w,v)\).  Hence
\(\langle V(\xi_w),V(\xi_v)\rangle
=q^{d(w,v)}=\langle\xi_w,\xi_v\rangle_q\).  Thus \(V\) extends
linearly to an isometric embedding on
\(\operatorname{span}\{\xi_w:w\in\mathcal W\}\); in particular, it
is isometric on
\(K_+[F]=\operatorname{span}\{\xi_w:w\in\mathcal W_+\}\).

We now compute \(\langle S\xi_w,\xi_v\rangle_q\) in this tensor
model.  For \(w\in\mathcal W\), let \(Sw\) denote the word obtained
by interchanging \(h_1\) and \(h_2\), so that
\(S\xi_w=\xi_{Sw}\).  Under \(V\), this interchange reverses the
coordinate indexed by \((h_1,h_2)\), leaves unchanged every
coordinate indexed by a pair contained in \(F\), and, for each
\(f\in F\), interchanges the two coordinates indexed by
\((f,h_1)\) and \((f,h_2)\).  We first compute its effect for a fixed
\(f\).

Let \(f\in F\) and \(w\in\mathcal W_+\). Since \(h_1\) precedes \(h_2\), the letter \(f\) has
exactly three possible positions: before \(h_1\), between \(h_1\) and
\(h_2\), or after \(h_2\). More precisely, the order of
\(f,h_1,h_2\) in \(w\) is one of
\(fh_1h_2\), \(h_1fh_2\), and \(h_1h_2f\).
In \(D_{(f,h_1)}\otimes D_{(f,h_2)}\), these three cases correspond
respectively to
\[
 a_0^+=d_0\otimes d_0,\qquad
 a_1^+=d_1\otimes d_0,\qquad
 a_2^+=d_1\otimes d_1.
\]
Let \(L_{+,f}\subset D_{(f,h_1)}\otimes D_{(f,h_2)}\) be their
span.  In the word \(Sw\), the corresponding orders are
\(fh_2h_1\), \(h_2fh_1\), and \(h_2h_1f\),
and their coordinates in the same space are
\[
 a_0^-=d_0\otimes d_0,\qquad
 a_1^-=d_0\otimes d_1,\qquad
 a_2^-=d_1\otimes d_1.
\]
Thus the action of \(S\) on these pair coordinates is the tensor flip
\(d_i\otimes d_j\mapsto d_j\otimes d_i\).

The Gram matrix \(A\) of \(a_0^+,a_1^+,a_2^+\) and the matrix
\(C=(\langle a_i^+,a_j^-\rangle)_{i,j=0}^2\) are
\[
 A=
 \begin{pmatrix}
  1&q&q^2\\
  q&1&q\\
  q^2&q&1
 \end{pmatrix},
 \qquad
 C=
 \begin{pmatrix}
  1&q&q^2\\
  q&q^2&q\\
  q^2&q&1
 \end{pmatrix}.
\]
Let \(S_f\) denote the compression to \(L_{+,f}\) of the tensor
flip on \(D_{(f,h_1)}\otimes D_{(f,h_2)}\).  If \([S_f]\) denotes its matrix in the basis
\(a_0^+,a_1^+,a_2^+\), then
\(C_{ij}=\langle a_i^+,S_fa_j^+\rangle=(A[S_f])_{ij}\).  Consequently,
\[
 [S_f]
 =A^{-1}C
 =
 \begin{pmatrix}
  1&q&0\\
  0&-q^2&0\\
  0&q&1
\end{pmatrix}.
\]
This matrix has eigenvalues \(1\) and \(-q^2\).

Let \(B_F\) be the tensor product of the pair spaces \(D_{(f,f')}\)
over \(f,f'\in F\) with \(f\prec f'\).  After regrouping the factors,
\(V(K_+[F])\) is contained in
\[
 \bbC d_0
 \otimes
 \left(\bigotimes_{f\in F}L_{+,f}\right)
 \otimes B_F.
\]
For \(w,v\in\mathcal W_+\), the coordinate indexed by
\((h_1,h_2)\) contributes \(\langle d_1,d_0\rangle=q\).  Each
\(f\in F\) contributes the local operator \(S_f\), while the pair
spaces indexed by pairs contained in \(F\) contribute the identity.
Therefore, for every \(\xi\in K_+[F]\),
\[
 \langle S\xi,\xi\rangle_q
 =q\left\langle
 \left[
 \mathrm{id}_{\bbC d_0}
 \otimes
 \left(\bigotimes_{f\in F}S_f\right)
 \otimes \mathrm{id}_{B_F}
 \right]
 V(\xi),
 V(\xi)
 \right\rangle.
\]
Since the operator in brackets is self-adjoint and has spectrum contained
in \([-q^2,1]\) (as each $ S_f $ has eigenvalues $1$ and $-q^2$), the result follows.
\end{proof}

\begin{remark}
In fact, the same construction in the proof above gives a general pair-order
realization: If \(f_1,\ldots,f_n\in H\) are orthonormal, then the
subspace
\[
 \operatorname{span}\{
 f_{\sigma(1)}\otimes\cdots\otimes f_{\sigma(n)}
 :\sigma\in S_n
 \}
 \subset\F
\]
embeds isometrically into \(D^{\otimes\binom n2}\), where
\(D=\bbC^2\) has unit vectors \(d_0,d_1\) satisfying
\(\langle d_0,d_1\rangle=q\).
\end{remark}

We now remove the distinctness assumption by replacing repeated letters with
distinct labelled copies following the ideas of  \cite[Section~1.7]{MS03}.

\begin{lemma}
\label{lem:lift-embedding}
Fix \(q\in(-1,1)\). Let \(h_1,\ldots,h_n\) be elements of an
orthonormal basis of \(H\), with repetitions allowed, and set
\(\xi_0=h_1\otimes\cdots\otimes h_n\). Let
\[
K(\xi_0)=\operatorname{span}\{h_{\sigma(1)}\otimes\cdots\otimes
h_{\sigma(n)}:\sigma\in S_n\}
\subset\mathcal F_q(H).
\]
Choose distinct orthonormal vectors
\(\widetilde h_1,\ldots,\widetilde h_n\) in \(\ell^2(\mathbb N)\).
For \(\rho\in S_n\) and
\(\xi=h_{\rho(1)}\otimes\cdots\otimes h_{\rho(n)}\), define
\[
\operatorname{Lift}(\xi)=
\left\{
\widetilde h_{\sigma(1)}\otimes\cdots\otimes\widetilde h_{\sigma(n)}
:\sigma\in S_n, 
h_{\sigma(1)}\otimes\cdots\otimes h_{\sigma(n)}=\xi
\right\}.
\]
Then
\[
\iota\xi=
|\operatorname{Lift}(\xi)|^{-1/2}
\sum_{\widetilde\xi\in\operatorname{Lift}(\xi)}
\widetilde\xi
\]
extends linearly to an isometric embedding
\(\iota:K(\xi_0)\to\mathcal F_q(\ell^2(\mathbb N))\).
\end{lemma}

\begin{proof}
Let \(\mathcal W\) be the set of words
\(w=(g_1,\ldots,g_n)\) obtained by permuting
\(h_1,\ldots,h_n\), and write
\(\xi_w=g_1\otimes\cdots\otimes g_n\). The vectors
\(\{\xi_w:w\in\mathcal W\}\) form a basis of \(K(\xi_0)\).

Notice that \(\operatorname{Lift}(\xi_w)\) and
\(\operatorname{Lift}(\xi_v)\) have the same cardinality for any
\(w,v\in\mathcal W\); denote this cardinality by \(r\).

For \(w=(g_1,\ldots,g_n)\) and
\(v=(g'_1,\ldots,g'_n)\) in \(\mathcal W\), set
\[
\operatorname{Match}(w,v)
=
\{\pi\in S_n:g_k=g'_{\pi(k)}\text{ for every }k\}.
\]
Note that
\[
\langle\xi_w,\xi_v\rangle_q
=
\sum_{\pi\in\operatorname{Match}(w,v)}
q^{\operatorname{inv}(\pi)}.
\]

Observe that, for fixed
\(\eta\in\operatorname{Lift}(\xi_w)\), matching the distinct labels
as \(\zeta\) ranges over \(\operatorname{Lift}(\xi_v)\) gives a
bijection with \(\operatorname{Match}(w,v)\). Hence
\[
\sum_{\zeta\in\operatorname{Lift}(\xi_v)}
\langle\eta,\zeta\rangle_q
=
\sum_{\pi\in\operatorname{Match}(w,v)}
q^{\operatorname{inv}(\pi)}
=
\langle\xi_w,\xi_v\rangle_q.
\]

Therefore,
\[
\langle \iota\xi_w,\iota\xi_v\rangle_q
=\frac1r\sum_{\eta\in\operatorname{Lift}(\xi_w)}
\sum_{\zeta\in\operatorname{Lift}(\xi_v)}
\langle\eta,\zeta\rangle_q
=\frac1r\sum_{\eta\in\operatorname{Lift}(\xi_w)}
\langle\xi_w,\xi_v\rangle_q
=\langle\xi_w,\xi_v\rangle_q.
\]
Thus \(\iota\) extends linearly to an isometric embedding from
\(K(\xi_0)\) into \(\mathcal F_q(\ell^2(\mathbb N))\).
\end{proof}

\begin{proof}[Proof of Proposition~\ref{prop:ordered-swap-estimate}]
Choose an orthonormal basis \(B\) of \(\{h_1,h_2\}^\perp\). By density, it suffices
to prove the estimate when \(\xi\in K_+\) is a finite linear
combination of tensors formed from \(h_1,h_2\) and elements of \(B\).

Collecting together the terms having the same multiset of
\(B\)-letters, write \(\xi=\xi_1+\cdots+\xi_n\). Thus, for each
\(j\), there are \(f_1,\ldots,f_m\in B\), repeated according to
multiplicity, such that
\[
 \xi_j\in
 K(h_1\otimes h_2\otimes f_1\otimes\cdots\otimes f_m),
\]
and different values of \(j\) correspond to different multisets.

Fix \(j\), and let \(\iota\) be the isometric embedding supplied by
Lemma~\ref{lem:lift-embedding}. Let \(\widetilde S\) interchange the labelled copies of
\(h_1,h_2\) and fix the remaining labels. By construction,
\(\iota S=\widetilde S\iota\). Lemma~\ref{lem:distinct-ordered-swap} therefore gives
\[
 \frac{\langle S\xi_j,\xi_j\rangle_q}{\|\xi_j\|_q^2}
 =
 \frac{\langle\widetilde S\iota\xi_j,\iota\xi_j\rangle_q}
      {\|\iota\xi_j\|_q^2}
 \in q[-q^2,1].
\]

Notice that, for \(j\neq k\),
\(\langle\xi_j,\xi_k\rangle_q
=\langle S\xi_j,\xi_k\rangle_q=0\). It follows that
\[
 \frac{\langle S\xi,\xi\rangle_q}{\|\xi\|_q^2}
 =
 \sum_{j=1}^n
 \frac{\|\xi_j\|_q^2}{\|\xi\|_q^2}
 \frac{\langle S\xi_j,\xi_j\rangle_q}{\|\xi_j\|_q^2}
 \in q[-q^2,1].
\]
\end{proof}

\printbibliography

\end{document}